%% file: main.tex
\documentclass[12pt]{amsart}

\usepackage{amsmath}
\usepackage{amsfonts}
\usepackage{amssymb}
\usepackage{amscd}
\usepackage{graphicx}
\usepackage[abbrev,alphabetic]{amsrefs}
\RequirePackage[dvipsnames,usenames]{color}
\usepackage{soul,xcolor}
\setstcolor{red}
\usepackage{stmaryrd}
\usepackage{mathtools}
\usepackage{booktabs}
\usepackage{multirow}
\usepackage{mathdots}
\newtagform{tiny}{\tiny(}{)}

\usepackage{mathtools}
\usepackage{hyperref}
\usepackage[margin=1.25in]{geometry}

\usepackage{amsthm}
\usepackage{comment}
\usepackage[all,cmtip]{xy}
\usepackage{tikz-cd}
\usetikzlibrary{cd}

\usepackage[all]{xy}

\makeatletter
\def\@tocline#1#2#3#4#5#6#7{\relax
  \ifnum #1>\c@tocdepth % then omit
  \else
    \par \addpenalty\@secpenalty\addvspace{#2}%
    \begingroup \hyphenpenalty\@M
    \@ifempty{#4}{%
      \@tempdima\csname r@tocindent\number#1\endcsname\relax
    }{%
      \@tempdima#4\relax
    }%
    \parindent\z@ \leftskip#3\relax \advance\leftskip\@tempdima\relax
    \rightskip\@pnumwidth plus4em \parfillskip-\@pnumwidth
    #5\leavevmode\hskip-\@tempdima
      \ifcase #1
       \or\or \hskip 1em \or \hskip 2em \else \hskip 3em \fi%
      #6\nobreak\relax
    \hfill\hbox to\@pnumwidth{\@tocpagenum{#7}}\par% <---- \dotfill -> \hfill
    \nobreak
    \endgroup
  \fi}
\makeatother

\newsavebox{\pullback}
\sbox\pullback{%
\begin{tikzpicture}%
\draw (0,0) -- (1ex,0ex);%
\draw (1ex,0ex) -- (1ex,1ex);%
\end{tikzpicture}}

\newsavebox{\pullbackdl}
\sbox\pullbackdl{%
\begin{tikzpicture}%
\draw (-1ex,0ex) -- (0ex,0ex);%
\draw (0ex,-1ex) -- (0ex,0ex);%
\end{tikzpicture}}

\newcommand{\rup}[1]{\lceil #1 \rceil}
\newcommand{\rdown}[1]{\lfloor #1 \rfloor}

\newcommand{\Q}{\mathbb{Q}}

\newcommand{\cHom}{\mathop{\mathcal{H}\! \mathit{om}}}

\DeclareMathOperator{\DDB}{\underline{\Omega}}

\newcommand{\bQ}{\mathbb{Q}}

\newcommand{\bZ}{\mathbb{Z}}

\newcommand{\cC}{\mathcal{C}}

\newcommand{\cH}{\mathcal{H}}

\newcommand{\cM}{\mathcal{M}}
\newcommand{\cO}{\mathcal{O}}

\newcommand{\sO}{\mathcal{O}}

\newcommand{\m}{\mathfrak{m}}
\newcommand{\n}{\mathfrak{n}}
\newcommand{\p}{\mathfrak{p}}

\DeclareMathOperator{\Supp}{Supp}
\DeclareMathOperator{\Spec}{Spec}

\DeclareMathOperator{\codim}{codim}

\DeclareMathOperator{\Ext}{Ext}

\DeclareMathOperator{\Exc}{Exc}

\def\phi{\varphi}
\def\epsilon{\varepsilon}

\newcommand*{\coloneq}{\mathrel{\mathop:}=}

\newcommand{\kdot}{{{\,\begin{picture}(1,1)(-1,-2)\circle*{2}\end{picture}\,}}}

\theoremstyle{plain}
\newtheorem{theorem}{Theorem}[section]

\newtheorem{proposition}[theorem]{Proposition}
\newtheorem{lemma}[theorem]{Lemma}
\newtheorem{corollary}[theorem]{Corollary}

\newtheorem{claim}[theorem]{Claim}
\newtheorem*{claim*}{Claim}

\newtheorem{theoremA}{Theorem}

\theoremstyle{definition}
\newtheorem{definition}[theorem]{Definition}

\newtheorem{example}[theorem]{Example}

\newtheorem*{setup*}{Setup}

\theoremstyle{remark}
\newtheorem{remark}[theorem]{Remark}



\numberwithin{equation}{theorem}
\title[Higher F-rational singularities]
{Higher F-rational singularities}
\author{Tatsuro Kawakami}
\address{Department of Mathematics, Graduate School of Science, Kyoto University, Kyoto 606-8502, Japan} 
\email{tatsurokawakami0@gmail.com}
\author{Jakub Witaszek} 
\address{Northwestern University, Department of Mathematics, Lunt Hall, 2033 Sheridan Road, Evanston, IL 60208, USA}
\email{jakub.witaszek@northwestern.edu}

\begin{document}

\begin{abstract}
We introduce higher $F$-rationality generalising $F$-rationality. We prove that a normal variety over a field of characteristic zero is $m$-rational if and only if it is $m$-$F$-rational after reduction modulo a sufficiently large prime $p$. Additionally, we establish new results on the logarithmic extension of forms.
\end{abstract}

\maketitle

\setcounter{tocdepth}{2}
\tableofcontents
\section{Introduction}

One of the key ideas that has been influencing the development of singularity theory is the interplay between complex and arithmetic settings. In this context, classical Du Bois and rational singularities in characteristic $0$ are related to $F$-injective and $F$-rational singularities in positive characteristic.

Motivated by recent developments in the study of higher Du Bois and higher rational singularities \cites{JKSY22,Mustata-Popa22, CDM24,MOPW23,SVV,Friedman-Laza2,Friedman-Laza1, popa2024injectivityvanishingdubois,Kawakami-Witaszek(ch=0),Park-Popa1,Park-Popa2}, which detect Hodge-theoretic phenomena, we introduced in \cite{Kawakami-Witaszek} higher $F$-injective singularities in positive characteristic. The goal of this article is to introduce higher $F$-rational singularities generalising the higher rational singularities of Friedman and Laza.

\begin{definition}[see Definition \ref{def:prekFinj-iso}]
Let $X$ be a $d$-dimensional normal variety defined over a perfect field $k$ of characteristic $p>0$.
Fix an integer $m \geq 0$. 
We say that $X$ is \emph{$m$-$F$-rational} if
\begin{enumerate}
\item $\mathrm{codim}_{X}\,\mathrm{Sing}(X) > 2m+1$,
\item $C \colon Z\Omega^{[i]}_{X} \to \Omega^{[i]}_{X}$ is surjective for all $i\leq m$, and\\[-0.9em]
\item for every closed point $\m \in X$ and every effective Cartier divisor $D$ on $X$, there exists an  integer $n_0>0$ such that the composition
\[
H^j_\m(\Omega^{[i]}_{X}) \xrightarrow{C_n^{-1}} H^j_\m\left(\frac{F^n_{*}\Omega^{[i]}_{X}}{B_n\Omega^{[i]}_{X}}\right) \xrightarrow{\mathrm{nat}} H^j_\m\left(\frac{F^n_{*}\Omega^{[i]}_{X}(D)}{B_n\Omega^{[i]}_{X}}\right)
\]
is injective for all $i\leq m$, $j \leq d - i$, and all $n \geq n_0$.
\end{enumerate}
\end{definition}
\noindent A review of Cartier operators may be found in Subsection \ref{ss:Cartier}. When Assumption (1) is dropped, we call $X$ \emph{pre-$m$-$F$-rational}. We refer to Example \ref{example2} and Example \ref{example3} for examples of $m$-$F$-rational singularities. 

Next, we provide equivalent definitions of higher $F$-rationality (cf.\ Corollary \ref{cor:correctDef}). 

\begin{theoremA} \label{thm:IndependenceIntro}
Fix an integer $m>0$ and let $X$ be a normal affine $d$-dimensional variety over a perfect field of  characteristic $p>0$. Assume that $C \colon Z\Omega^{[i]}_{X} \to \Omega^{[i]}_{X}$ is surjective for all $i\leq m$. Then the following statements are equivalent:
\begin{enumerate}
\item $X$ is pre-$m$-$F$-rational,
\item there exists a Cartier divisor $D \geq 0$ on $X$ containing ${\rm Sing}(X)$ such that for every closed point $\m \in X$ and every $r>0$ there exists $n_0>0$ such that the composition
\[
H^j_\m(\Omega^{[i]}_{X}) \xrightarrow{C_n^{-1}} H^j_\m\left(\frac{F^n_{*}\Omega^{[i]}_{X}}{B_n\Omega^{[i]}_{X}}\right) \xrightarrow{\mathrm{nat}} H^j_\m\left(\frac{F^n_{*}\Omega^{[i]}_{X}(rD)}{B_n\Omega^{[i]}_{X}}\right)
\]
is injective when $0 \leq i\leq m$, $i+j \leq d$, and $n \geq n_0$.
\item $X$ is $F$-rational and for every closed point $\m \in X$ there exists an integer $n_0>0$ such that the composition
\[
H^j_\m(\Omega^{[i]}_{X}) \xrightarrow{C_n^{-1}} H^j_\m\left(\frac{F^n_{*}\Omega^{[i]}_{X}}{B_n\Omega^{[i]}_{X}}\right) \xrightarrow{\delta} H^{j+1}_\m(B_n\Omega^{[i]}_X) 
\]
is injective when $0 < i \leq m$,  $i+j \leq d$, and $n \geq n_0$.
\item $X$ is $F$-rational and for every closed point $\m \in X$ there exists an integer $n_0>0$ such that the map:
\[
C_n \colon H^j_\m(Z_n\Omega^{[i]}_{X}) \to   H^j_\m(\Omega^{[i]}_{X}) 
\]
is zero when $0 < i \leq m$,  $i+j \leq d$, and $n \geq n_0$.
\end{enumerate}
\end{theoremA}
\noindent The map $\delta$ in (3) of the above theorem comes from the long exact sequence:
\[
\ldots \to H^j_\m(B_n\Omega^{[i]}_{X}) \to H^j_\m(F^n_{*}\Omega^{[i]}_{X}) \to H^j_\m\left(\frac{F^n_{*}\Omega^{[i]}_{X}}{B_n\Omega^{[i]}_{X}}\right) \xrightarrow{\delta} H^{j+1}_\m(B_n\Omega^{[i]}_X) \to \ldots
\]
induced by $0 \to B_n\Omega^{[i]}_X \to F^n_{*}\Omega^{[i]}_{X} \to \frac{F^n_{*}\Omega^{[i]}_{X}}{B_n\Omega^{[i]}_{X}} \to 0$. We also note that $m$-$F$-rationality can be equivalently defined through the vanishing of $H^j_\m(Z_n\Omega^{[i]}_{X})$ (see Remark \ref{remark:ZZeroDef}). \\

A key property of $m$-rational singularities is that they impose depth conditions on reflexive differential forms, namely, $H^j_\m(\Omega^{[i]}_X) = 0$ for $i \leq m$ and $i+j < \dim X$. We establish the same result for $m$-$F$-rational singularities.

\begin{theoremA}[{Theorem \ref{thm:vanishing of local cohomologies}}] \label{thm:vanishingIntro}
Fix an integer $m \geq 0$ and let $X$ be a normal affine variety with $m$-$F$-rational singularities over a perfect field of characteristic $p>0$. Then
\begin{equation} \label{eq:vanishingmFrational}
H^j_{\m}(\Omega^{[i]}_X)=0 
\end{equation}
for all closed points $\m \in X$ when $i\leq m$ and $i+j<\dim X$.   
\end{theoremA}

Our main result compares higher rational singularities in characteristic zero and higher $F$-rational singularities in positive characteristic. In contrast to higher $F$-injectivity we do not need to assume the ordinarity conjecture and the statement holds for all sufficiently large $p>0$ (cf.\ \cite[Theorem A]{Kawakami-Witaszek}). This theorem generalises the results on classical $F$-rationality from \cite{Hara98}, \cite{Mehta-Srinivas(rationalsingularities)}, and \cite{Smith(rational)}.
\begin{theoremA}[{Theorem \ref{thm:k-F-rational to k-rational} and Theorem \ref{thm:pre-k-F-rational type to pre-$k$-rational}}]
    Let $X$ be a normal variety over a field $K$ of characteristic zero. Fix $m\geq 0$ and a model $\mathcal X$ of $X$ over a finitely generated $\bZ$-subalgebra $A \subseteq K$.
    \begin{enumerate}
        \item Suppose that $X$ is $m$-rational. Then there exists a Zariski-open set of closed points $S \subseteq \Spec A$ such that $X_s$ is $m$-$F$-rational for all $s \in S$.
        \item The converse also holds true: if $X_s$ is $m$-$F$-rational for a Zariski-dense set of closed points  $S \subseteq \Spec A$, then $X$ is $m$-rational.
    \end{enumerate}
\end{theoremA}
 For $K=\bQ$, the theorem implies that
\[
\text{ $X$ is $m$-rational }\!\! \iff\!\! \text{ $X \bmod p$ is $m$-$F$-rational for all (equivalently, some) $p \gg 0$.}
\]
\vspace{-0.5em}

A key result about rational singularities $X$ in characteristic $0$ is that they satisfy the \emph{extension theorem}: $\pi_*\Omega^i_Y = \Omega^{[i]}_X$ for $\pi \colon Y \to X$ being a resolution of singularities (see Kebekus--Schnell's \cite{KS21}). Unfortunately, the extension theorem breaks in characteristic $p>0$ even for basic toric $A_p$-surface singularities (see \cite{Langer19,Gra}). 

On the other hand, one can ask if the logarithmic variant of extension theorem holds for $F$-singularities. 
In \cite{Kaw4}, the first author have showed that $F$-rational singularities with ${\rm codim}\, {\rm Sing}(X) \geq 3$ satisfy the logarithmic extension theorem for $1$-forms. We extend this result to  $m$-forms when $X$ has $(m-1)$-$F$-rational singularities.

\begin{theoremA}
    Let $X$ be a normal variety over a perfect field of positive characteristic with $(m-1)$-$F$-rational singularities for $m\geq 2$.
    Then 
    \[
    f_{*}\Omega^{[m]}_Y(\log E)\cong \Omega^{[m]}_X
    \]
    for all proper birational morphisms $f\colon Y\to X$ with $Y$ normal and $E$ being the reduced exceptional divisor.
\end{theoremA}

\subsection{Acknowlegements}
The authors thank Bhargav Bhatt, Brad Dirks, Eamon Quinlan-Gallego, Mircea Musta{\c{t}}{\u{a}}, Mihnea Popa, Supravat Sarkar, Wanchun Shen, Karl Schwede,  Teppei Takamatsu, Sridhar Venkatesh, Duc Vo, Shou Yoshikawa, and Bogdan Zavyalov for valuable conversations related to the content of the paper.
Kawakami was supported by JSPS KAKENHI Grant number JP24K16897 and by Inamori Foundation.
Witaszek was supported by NSF research grants DMS-2101897 and DMS-2401360.

\section{Preliminaries}
\subsection{Notation and terminology}
In this paper, we use the following notation:
\begin{itemize}
\item A \textit{variety} is an integral separated scheme of finite type over a field. 
\item Given a proper birational morphism $f\colon Y\to X$, we denote by $\Exc(f)$ the reduced exceptional divisor. 
\item Given an integral normal Noetherian scheme $X$, a proper birational morphism $\pi \colon Y \to X$ is called \emph{a log resolution of $X$} if $Y$ is regular and $\mathrm{Exc}(\pi)$ is a simple normal crossing divisor. A log resolution is \emph{strong} if it is an isomorphism over the regular locus of $X$.
\item For a Noetherian scheme $X$, we denote by $D^b_{\mathrm{coh}}(X)$ the derived category consisting of bounded complexes of coherent $\cO_X$-modules.
\item For  $\mathcal{F}^{\bullet}\in D^b_{\mathrm{coh}}(X)$, we denote the canonical truncation of $\mathcal{F}^{\bullet}$ by $\tau^{>j}\mathcal{F}^{\bullet}$ (cf.\ \cite[\href{https://stacks.math.columbia.edu/tag/0118}{Tag 0118}]{stacks-project}). For a proper birational morphism $f\colon Y\to X$ of Noetherian schemes,
we denote $\tau^{>j}Rf_{*}\mathcal{F}$ by $R^{>j}f_{*}\mathcal{F}$.
\item 
Given an integral normal Noetherian scheme $X$ and a $\bQ$-divisor $D$, 
we define the subsheaf $\sO_X(D)$ of the sheaf $K(X)$ of rational functions on $X$ 
by the following formula
\[
\Gamma(U, \sO_X(D)) = 
\{ \varphi \in K(X) \mid 
\left({\rm div}(\varphi)+D\right)\!|_U \geq 0\}
\]
{for every} open subset $U$ of $X$. 
In particular, 
$\cO_X({D}) = \cO_X(\rdown{D})$. Moreover, we denote the fractional part of a $\Q$-divisor $D$ by $\{D\}\coloneqq D-\lfloor D \rfloor$.
\item We let $\DDB^i_X \in D^b_{\rm coh}(X)$ denote the $i$-th Du Bois complex of a variety $X$ over a field of characteristic $0$. When $X$ is smooth, one has $\DDB^i_X=\Omega^i_X[0]$. In general, it is given by
\[
\DDB^i_X \coloneqq R\epsilon_{\kdot,*}\Omega^i_{X_\kdot},
\]
where $\epsilon_\kdot \colon X_{\kdot}\to X$ is a hyperresolution (see \cite[Chapter~5]{GNPP} or \cite[Chapter~7.3]{Peter-Steenbrink(Book)} for details). Equivalently, $\DDB^i_X$ is the derived $h$-sheafification of $\Omega^i$ on the $h$-site over $X$. It is a well-known fact that $\cH^0(\DDB^i_X)$ is torsion-free (see \cite[Proposition 4.2 (i)]{Huber-Jorder}). In this paper, we denote $\cH^0(\DDB^i_X)$ by $\Omega^i_{X,h}$.
\item Let $X$ be a $d$-dimensional normal affine variety over a perfect field $k$ of characteristic $p>0$.
We say $X$ is \emph{$F$-rational} if for every closed point $\m \in X$ and every Cartier divisor $D\geq 0$ on $X$, there exists a positive integer $n_0>0$ such that the composition
\[
H^j_\m(\sO_{X}) \xrightarrow{F^n} H^j_\m\left(F^n_{*}\sO_X\right) \xrightarrow{\mathrm{nat}} H^j_\m\left(F^n_{*}\sO_{X}(D)\right)
\]
is injective for all $0\leq j\leq d$ and $n \geq n_0$ (cf.~\cite[Proposition 3.9]{Takagi-Watanabe}).
\end{itemize}

We start with an easy criterion for when a coherent sheaf is ($S_k$).
\begin{lemma} \label{lem:reflexive}
Let $X = \Spec R$ be a normal affine variety over a field $k$ and let $\cM$ be a coherent sheaf on $X$. Let $U \subseteq X$ be an open subset such that $\cM|_U$ is locally free and $Z \coloneqq X\smallsetminus U$ is of codimension at least two. Then $\cM$ is reflexive if and only if for every prime ideal $\p \in Z$ there exists a maximal ideal $\m \in \overline{\p} \coloneqq V(\p)$ such that
\begin{equation} \label{eq:vanrefl}
H^{\dim \overline{\p}}_\m(\cM) = H^{\dim \overline{\p}+1}_\m(\cM) = 0. 
\end{equation}
\end{lemma}
\noindent Equivalently, one can ask for an existence of an open subset $U \subseteq X$ containing $\p$ and such that \eqref{eq:vanrefl} holds for every $\m \in U$.
\begin{proof}
The sheaf $\cM$ is reflexive if and only if
\[
H^0_\p(\cM_{\p}) = H^1_\p(\cM_{\p}) = 0
\]
for every prime ideal $\p \in Z$. By local duality, this is equivalent to checking that
\[
{\rm Ext}^{-r}(\cM_\p, \omega^\bullet_{\cO_{X,\p}}) = {\rm Ext}^{-r-\dim \overline{p}}(\cM, \omega^\bullet_{X})_\p 
\]
is zero for $r=0$ and $r=-1$. This is equivalent to verifying that there exists an open subset $U \subseteq X$ containing $\p$ such that
\[
{\rm Ext}^{-r-\dim \overline{p}}(\cM, \omega^\bullet_{X})|_U = 0,
\]
which is the same as finding a maximal ideal $\m \in \overline{p}$ such that
${\rm Ext}^{-r-\dim \overline{p}}(\cM, \omega^\bullet_{X})_\m=0.$
By local duality again, this translates to:
\[
H^{\dim \overline{\p}}_\m(\cM) = H^{\dim \overline{\p}+1}_\m(\cM) = 0
\]
as required. \qedhere
\end{proof}

\begin{remark}
    \label{rem:Sk}
The same argument as above shows the following statement.
Let $X = \Spec R$ be a normal affine variety over a field and let $\cM$ be a coherent sheaf on $X$. 
Let $Z\subseteq X$ be a closed subset such that $\mathcal{M}|_{X\setminus Z}$ is locally-free and let $k \geq 0$ be an integer such that
\[
H^i_\m(\cM) = 0 
\]
for all closed points $\m \in Z$ and $i < \dim\,Z+k$. Then $\cM$ satisfies ($S_{k}$).
\end{remark}

In this paper, we will also repeatedly use the following lemma.

\begin{lemma}[{\cite[Lemma 4.1]{Hara98}}]\label{lem:Hara's lemma}
   Let $X$ be a Noetherian separated scheme of finite type over a Noetherian ring $A$ and let $f\colon X\to \Spec\,A$ be a natural projection. Let $\mathcal{F}$ be a quasi-coherent sheaf on $X$ such that $R^jf_{*}\mathcal{F}$ is flat over $A$ for every $j\geq 0$.
   Then the natural map
   \[
   R^jf_{*}\mathcal{F}\otimes_A k(s)\xrightarrow{\cong} R^jf_{s,*}\mathcal{F}_s
   \]
   is an isomorphism for every point $s\in\Spec A$ and $j\geq 0$, where $k(s)$ is the residue field of $s\in\Spec A$, $X_s=X\times_A \Spec k(s)$, and $\mathcal{F}_s$ is the induced sheaf on $X_s$.
\end{lemma}

\subsection{Higher rational singularities}
 In this subsection, we recall the definition and basic properties of higher rational singularities.
\begin{definition}[{\cite{SVV}}]
    Let $X$ be a normal variety over an algebraically closed field of characteristic zero.
    We say that $X$ is \emph{pre-$m$-rational} if 
    \[
    R^{>0}\cHom(\underline{\Omega}^{d-i}_X,\omega^{\bullet}_X[-d])=0
    \]
    for all $i\leq m$.
    Moreover, we say that $X$ is \emph{$m$-rational} if it is pre-$m$-rational and $\mathrm{codim}_{X}\, {\rm Sing}(X) > 2m+1$.
\end{definition}

\begin{proposition}\label{prop:characterization of pre-k-rational}
    Let $X$ be a normal variety over an algebraically closed field of characteristic zero.
    Suppose that $\codim_X {\rm Sing}(X)>m$.
    Then $X$ is pre-$m$-rational if and only if 
    \[
    R^{>0}\pi_{*}\Omega^{i}_Y(\log E)=0
    \]
    for every $i\leq m$ and every (equivalently, some) strong log resolution $\pi\colon Y\to X$ with reduced exceptional divisor $E$.
\end{proposition}
\begin{proof}
    This has been established in \cite{SVV} (see also \cite[Proposition 2.26]{Kawakami-Witaszek(ch=0)}).
\end{proof}

\noindent The following characterisation of higher rationality is one of the key results of \cite{Kawakami-Witaszek(ch=0)}.

\begin{theorem}\label{thm:characterization of pre-k-rational}
    Let $X$ be a normal connected variety of dimension $d$ over an algebraically closed field of characteristic zero and let $\pi \colon Y \to X$ be a log resolution with reduced exceptional divisor $E$.
    Then $X$ is pre-$m$-rational if and only if 
    \[
    H_{\m}^j(\Omega^i_{X,h})\to H_{\m}^j(R\pi_{*}\Omega^i_Y(\log E))
    \]
    is injective for all $i, j \geq 0$ such that $i\leq m$ and $i + j\leq d$.
\end{theorem}
\begin{proof}
    See \cite[Corollary 4.10]{Kawakami-Witaszek(ch=0)}.
\end{proof}

\subsection{Cartier operators} \label{ss:Cartier}
Let $Y$ be a smooth variety and let $E$ be a reduced divisor on $Y$ with snc support.
Let $A$ be a $\Q$-divisor whose support of the fractional part $\{A \} \coloneqq A-\lfloor A\rfloor$ is contained in $E$. 

Consider the complex:
\[
F_*\Omega_Y^\bullet(\log E)(pA) \coloneqq F_*\left(\Omega_Y^{\bullet}(\log E)\otimes \sO_Y(\lfloor pA \rfloor)\right).
\]
We define $\sO_Y$-modules $B\Omega_Y^i(\log E)(pA)$ and $Z\Omega_Y^i(\log E)(pA)$ to be the boundaries and the cycles of  this complex at $i$: 
\begin{align} 
\label{eq:definition-of-B1} B \Omega_Y^i(\log E)(pA) &\coloneqq {\rm im}\left(F_*\Omega_Y^{i-1}(\log E)(pA) \xrightarrow{F_*d}  F_*\Omega_Y^{i}(\log E)(pA)\right),\\ 
Z\Omega_Y^i(\log E)(pA) &\coloneqq {\rm ker}\left(F_*\Omega_Y^{i}(\log E)(pA) \xrightarrow{F_*d}  F_*\Omega_Y^{i+1}(\log E)(pA)\right) \nonumber,
\end{align}
We emphasise that, in general
\begin{align*}
    F_*\Omega_Y^i(\log E)(pA)&\not \cong F_*\Omega_Y^i(\log E)\otimes\sO_Y(pA), \\
    B \Omega_Y^i(\log E)(pA)&\not \cong B \Omega_Y^i(\log E)\otimes\sO_Y(pA),\\
    Z\Omega_Y^i(\log E)(pA)&\not \cong Z\Omega_Y^i(\log E)\otimes\sO_Y(pA).
\end{align*}
By definition, we obtain the short exact sequence
\begin{equation}\label{eq:ZOmegaB,single}
    0\to Z\Omega_Y^i(\log E)(pA) \to F_{*}\Omega_Y^i(\log E)(pA) \xrightarrow{F_*d} B\Omega_Y^{i+1}(\log E)(pA) \to 0.
\end{equation}
Moreover, the Cartier isomorphism (\cite[Lemma 3.3]{Hara98}, \cite[equation (5.4.2)]{KTTWYY1}) gives the short exact sequence
\begin{equation} \label{eq:BZOmega,single}
0 \to B\Omega^i_Y(\log E)(pA) \to Z\Omega^i_Y(\log E)(pA) \xrightarrow{C} \Omega^i_Y(\log E)(A) \to 0,
\end{equation}
which allows for the construction of the inverse Cartier operator
\begin{align}
\label{eq:dual-Cartier-Hara} C^{-1}\colon \Omega^i_Y(\log E)(A)&\underset{\cong}{\xleftarrow{\,  C\, }} \frac{Z\Omega^{i}_Y(\log E)(pA)}{B\Omega^{i}_Y(\log E)(pA)} \\[0.2em]
&\xhookrightarrow{\hphantom{\, C\, }} \frac{F_*\Omega^{i}_Y(\log E)(pA)}{B\Omega^{i}_Y(\log E)(pA)} \nonumber 
 \end{align}
 
\subsubsection{Iterations}
We refer to \cite[Section 5.2]{KTTWYY1} for a more detailed discussion on higher Cartier operators. 
Inductively on $n$, one constructs locally free $\sO_Y$-submodules \[
B_n\Omega_Y^i(\log E)(p^nA)\quad\text{and}\quad
Z_n\Omega_Y^i(\log E)(p^nA) 
\] 
of $F^n_* \Omega^i_Y(\log E)(p^nA)$ with the following short exact sequences of $\sO_Y$-modules:\\[-0.5em]
{ \small 
\makeatletter
\renewcommand{\mintagsep}{0em}
\renewcommand{\minalignsep}{0em}
\makeatother
\thinmuskip=0.55mu
\medmuskip=1.1mu
\thickmuskip=1.65mu
\begin{align}
&0\to Z_n\Omega^i_Y(\log E)(p^{n}A) \to F_{*}Z_{n{-}1}\Omega^i_Y(\log E)(p^nA) \xrightarrow{(\star)} B\Omega^{i{+}1}_Y(\log E)(pA) \to 0,\label{ZZB}\\[0.6em]
&0 \to B_{n}\Omega_Y^i(\log E)(p^nA) \to Z_{n}\Omega_Y^i(\log E)(p^nA) \xrightarrow{C_{n}}  \Omega_Y^i(\log E)(A) \to 0\label{BZOmega},\\[0.6em]
&0\to {F^{n{-}1}_{*}}\hspace{-0.2em}B\Omega^i_Y(\log E)(p^{n}\hspace{-0.2em}A) \to Z_{n}\Omega^i_Y(\log E)(p^{n}\hspace{-0.2em}A) \xrightarrow{C} Z_{n{-}1}\Omega^i_Y(\log E)(p^{n{-}1}\hspace{-0.2em}A) \to 0,\label{BZZ}\\[0.6em]
&0\to {F^{n{-}1}_{*}}\hspace{-0.25em}B\Omega_Y^i(\log E)(p^{n}\hspace{-0.2em}A) \to B_{n}\Omega_Y^i(\log E)(p^{n}\hspace{-0.2em}A) \xrightarrow{C} B_{n{-}1}\Omega_Y^i(\log E)({p^{n{-}1}}\hspace{-0.25em}A) \to 0,\label{eq:BBB} \\[-1em] \nonumber
\end{align}}
\!\!for all $n \geq 1$ where $(\star)$ denotes the map $F_{*}d\circ F_{*}C_{n-1}$. We refer to \cite[Subsection 5.2]{KTTWYY1} (see also \cite[Section 3]{Kaw7}) for the construction of these short exact sequences.  
Similarly to \eqref{eq:BBB}, we also have the short exact sequence
{ \small 
\makeatletter
\renewcommand{\mintagsep}{0em}
\renewcommand{\minalignsep}{0em}
\makeatother
\thinmuskip=0.55mu
\medmuskip=1.1mu
\thickmuskip=1.65mu\begin{equation}\label{eq:BBB2}
  0\to F_{*}B_{n-1}\Omega_Y^i(\log E)(p^{n}A) \to B_{n}\Omega_Y^i(\log E)(p^{n}A) \xrightarrow{C_{n-1}} B\Omega_Y^i(\log E)(pA) \to 0  
\end{equation}
}

When $n=1$, equations \eqref{ZZB} and \eqref{BZOmega} are  nothing but \eqref{eq:ZOmegaB,single} and \eqref{eq:BZOmega,single}, respectively, where implicitly:
\begin{align*}
    Z_0\Omega_Y^i(\log E)(A)&\coloneqq \Omega_Y^i(\log E)(A)\\
    Z_1 \Omega_Y^i(\log E)(pA)&\coloneqq Z\Omega_Y^i(\log E)(pA),\\
    B_0\Omega_Y^i(\log E)(A)&\coloneqq 0,\\
    B_1 \Omega_Y^i(\log E)(pA)&\coloneqq B\Omega_Y^i(\log E)(pA).
\end{align*}
Note that 
\[
\frac{F^n_*\Omega^{i}_Y(\log E)(p^nA)}{B_n\Omega^{i}_Y(\log E)(p^nA)}
\]
is locally free for all $n$ (see the paragraph immediately after \cite[equation (3.0.20)]{Kaw7}).
We also have the iterated inverse Cartier operator
\begin{align}
\label{eq:dual-Cartier-Hara-higher} 
C^{-1}_{n}\colon \Omega^i_Y(\log E)(A)&\underset{\cong}{\xleftarrow{\,C_{n}\,}} \frac{Z_n\Omega^{i}_Y(\log E)(p^nA)}{B_n\Omega^{i}_Y(\log E)(p^nA)} \\[0.2em]
&\xhookrightarrow{\hphantom{\,C_{n}\, }} \frac{F^n_*\Omega^{i}_Y(\log E)(p^nA)}{B_n\Omega^{i}_Y(\log E)(p^nA)} \nonumber  \end{align}

\subsubsection{Reflexive Cartier operators}
In what follows, we assume that $X$ is a normal variety over a perfect field $k$ of positive characteristic, and let $j\colon U\hookrightarrow X$ be the inclusion of the smooth locus.
We set 
\[
B_n \Omega_X^{[i]} \coloneqq j_{*}B_{n}\Omega_U^{i}\hspace{5mm}\text{and}\hspace{5mm} Z_n\Omega_X^{[i]} \coloneqq j_{*}Z_{n}\Omega_U^{i},
\]
which are reflexive $\sO_X$-modules.
By pushing forward the short exact sequence \eqref{BZOmega}
\[
0 \to B_{n}\Omega^i_U \to Z_{n}\Omega^i_U \xrightarrow{C_{n}}  \Omega^{i}_U \to 0, 
\]
we obtain an exact sequence
\[
0 \to B_{n}\Omega^{[i]}_X \to Z_{n}\Omega^{[i]}_X \xrightarrow{C_{n}}  \Omega^{[i]}_X.
\]
The surjectivity of $C\colon Z\Omega^{[i]}_X\to \Omega^{[i]}_X$ is related to the logarithmic extension property of differential forms as indicated by the following result (see \cite{Kaw4} for details).

\begin{theorem}\label{thm:LET}
    Let $X$ be a normal variety over a perfect field $k$ of positive characteristic and
    let $f\colon Y\to X$ be a proper birational morphism with reduced exceptional divisor $E$.
    Fix an integer $m\geq 0$.
    If 
    \[
    C \colon Z\Omega^{[m]}_X \to  \Omega^{[m]}_X.
    \]
    is surjective, then the following statements hold:
    \begin{enumerate}
        \item $f_{*}\Omega^{[m]}_Y(\log E)\cong \Omega^{[m]}_X$, and \\[-0.7em]
        \item $f_{*}B_n\Omega^{
        [m]}_Y(\log E)\cong B_n\Omega^{[m]}_X$ for all $n\geq 0$.
\end{enumerate}
\end{theorem}
\begin{proof}
    By \cite[Theorem A]{Kaw4}, we have (1).
    Consider the short exact sequence
    \[
    0\to B_n\Omega^{[m]}_Y(\log E)\to F^n_{*}\Omega^{[m]}_Y(\log E) \to \frac{F^n_{*}\Omega^{[m]}_Y(\log E)}{B_n\Omega^{[m]}_Y(\log E)}\to 0.
    \]
    By pushing forward by $f$ and using (1), we get an exact sequence
    \[
    0\to f_{*}B_n\Omega^{[m]}_Y(\log E)\to F^n_{*}\Omega^{[m]}_X \to f_{*}\left(\frac{F^n_{*}\Omega^{[m]}_Y(\log E)}{B_n\Omega^{[m]}_Y(\log E)}\right).
    \]
    Since $\Omega^{[m]}_Y(\log E)$ and $B\Omega^{[m]}_Y(\log E)$ are reflexive, the sheaf
    \[
    \frac{F^n_{*}\Omega^{[m]}_Y(\log E)}{B_n\Omega^{[m]}_Y(\log E)} \quad \text{ and so also } \quad f_{*}\left(\frac{F^n_{*}\Omega^{[m]}_Y(\log E)}{B_n\Omega^{[m]}_Y(\log E)}\right) 
    \]
    are torsion-free. 
    This shows that $f_{*}B_n\Omega^m_Y(\log E)$ is reflexive.
\end{proof}

Suppose that $C \colon Z\Omega^{[i]}_{X} \to \Omega^{[i]}_{X}$ is surjective.
Then $C_n \colon Z_n\Omega^{[i]}_{X} \to \Omega^{[i]}_{X}$ is surjective for every $n\geq 0$ \cite[Lemma 2.9]{Kawakami-Sato}, 
and we have the iterated inverse Cartier operator
\begin{equation}\label{eq:inverse reflexive Cartier op}
    C^{-1}_n\colon \Omega^{[i]}_X \underset{\cong}{\xleftarrow{\,C_{n}\,}} \frac{Z_n\Omega^{[i]}_{X}}{B_n\Omega^{[i]}_{X}} \xhookrightarrow{\hphantom{ C_n }} \frac{F^n_{*}\Omega^{[i]}_{X}}{B_n\Omega^{[i]}_{X}}. %=:  G_n\Omega^{[i]}_X.
\end{equation}

\subsubsection{Vanishing theorem}
In this subsection we establish some key vanishing results.

\begin{lemma} \label{lemma:Kawakami-vanishing}
Let $X$ be a normal variety over a perfect field of  characteristic $p>0$. Let $\pi \colon Y \to X$ be a log resolution of $X$ with exceptional divisor $E$. 
Fix $m\geq 0$ and suppose that
\[
R^{>0}\pi_{*}\Omega^i_Y(\log E)=0
\]
for all $i\leq m-1$.
Then the following statements hold:
\begin{enumerate}
    \item $R^{>0}\pi_{*}Z_n\Omega^i_Y(\log E)=0$ for all $i\leq m-1$ and $n\geq 0$.\\[-0.7em]
    \item $R^{>0}\pi_{*}B_n\Omega^i_Y(\log E)=0$ for all $i\leq m$ and $n\geq 0$.\\[-0.7em]
    \item For any fixed $i\leq m$, the Cartier  operator 
    $C \colon Z\Omega^{[i]}_X\to \Omega^{[i]}_X$ is surjective if and only if $\pi_{*}\Omega^i_Y(\log E) = \Omega^{[i]}_X$.
\end{enumerate}
\end{lemma}
\noindent In particular, if $R\pi_*\Omega^i_Y(\log E) = \Omega^{[i]}_X$ for every $i \leq m$, then $C \colon Z\Omega^{[i]}_X \to \Omega^{[i]}_X$ is surjective and $R\pi_*B\Omega^i_Y(\log E) = B\Omega^{[i]}_X$ for every $i \leq m$ (see Theorem \ref{thm:LET}).
\begin{proof}
We first prove (1) and (2) by induction on $m$.
By inductive hypothesis, we may assume that $R^{>0}\pi_{*}B_n\Omega^{m-1}_Y(\log E)=0$ for all $n\geq 0$.
From the short exact sequences
\[
0 \to B_n\Omega^{m-1}_Y(\log E)\to Z_n\Omega^{m-1}_Y(\log E)\xrightarrow{C_n} \Omega^{m-1}_Y(\log E)\to 0,
\]
we obtain $R^{>0}\pi_{*}Z_n\Omega^{m-1}_Y(\log E)=0$ for all $n\geq 0$.
From the short exact sequences
\begin{align*}
    & 0 \to Z\Omega^{m-1}_Y(\log E)\to F_{*}\Omega^{m-1}_Y(\log E)\to B\Omega^{m}_Y(\log E)\to 0.\\
    & 0 \to F^{n-1}_*B\Omega^{m}_Y(\log E)\to B_n\Omega^{m}_Y(\log E)\to B_{n-1}\Omega^{m}_Y(\log E)\to 0.
\end{align*}
we obtain $R^{>0}\pi_{*}B_n\Omega^m_Y(\log E)=0$ for all $n\geq 0$.

Finally, we show (3). Fix $i\leq m$.
The only if part is Theorem \ref{thm:LET}.
The if part follows from (2) and the short exact sequence
\[
0 \to B\Omega^{i}_Y(\log E)\to Z\Omega^{i}_Y(\log E)\xrightarrow{C} \Omega^{i}_Y(\log E)\to 0. \qedhere
\] 
\end{proof}

\subsection{Modulo $p$-reductions}

In this subsection, we discuss how local cohomology reduces modulo $p$. We start with the following lemma.

\begin{lemma} \label{lem:reductionmodp-basic}
Let $f \colon A \to R$ be a map of rings and let 
\[
M \xrightarrow{\alpha} N \xrightarrow{\beta} K
\]
be a complex of $R$-modules supported in degrees $[-1,1]$. Assume that $A$ is a Noetherian domain, $R$ is a finitely generated $A$-algebra, and $M, N, K$ are finite modules over $R$.

Then there exists an open subset $U \subseteq \Spec A$ such that for every closed point $s \in \Spec A$, the natural map
\begin{equation} \label{eq:tensor-goal}
\cH^0(M \xrightarrow{\alpha} N \xrightarrow{\beta} K) \otimes_R R_s \longrightarrow  \cH^0(M_s \xrightarrow{\alpha_s} N_s \xrightarrow{\beta_s} K_s)
\end{equation}
is an isomorphism. 
\end{lemma}
\noindent Here $R_s \coloneqq R \otimes_A k(s)$, $M_s \coloneqq M \otimes_A k(s)$,  $N_s \coloneqq M \otimes_A k(s)$,  and $K_s \coloneqq K \otimes_A k(s)$.
\begin{proof}
Note that $(-) \otimes_R R_s$ agrees with $(-) \otimes_A k(s)$. Namely both of these functor coincide with quotienting by the ideal $I_s \coloneqq {\rm ker}(A \to k(s))$.

By generic freeness \cite[Tag 051S]{stacks-project}, we can find an affine open subset $U\coloneqq \Spec A_g \subseteq \Spec A$ such that ${\rm ker}(\beta)/{\rm im}(\alpha)$, ${\rm im}(\beta)$, and $K/{\rm im}(\beta)$ are all free over $A_g$. Up to replacing $A$ by $A_g$, we may assume that these are free over $A$. 

We will now show that (\ref{eq:tensor-goal}) holds for every closed point $s \in \Spec A$. Consider the following diagram:
\[
\begin{tikzcd}[column sep = small]
& & 0 \ar{r} & {\rm im}(\beta) \otimes_R R_s \ar{r} & K_s \ar{r}  & \left(K/{\rm im}(\beta)\right) \otimes_R R_s \ar{r} & 0 \\
0 \ar{r} & {\rm ker}(\beta) \otimes_R R_s \ar{r} & N_s \ar{r}  &{\rm im}(\beta) \otimes_R R_s \ar{u}{=}  \ar{r} & 0 &  & \\
M_s \ar[two heads]{r}{(\star\star)} & {\rm im}(\alpha) \otimes_R R_s. \ar[hook]{u}{(\star)} & &  &
\end{tikzcd}
\]
Here the top row is a short exact sequence by ${\rm Tor}^1$ calculation, because $K/{\rm im}(\beta)$ is free over $A$. Similarly, the middle row is a short exact sequence, because ${\rm im}(\beta)$ is free over $A$. Likewise, the map $(\star)$ is an injection, because ${\rm ker}(\beta)/{\rm im}(\alpha)$ is free over $A$. Finally, the map $(\star\star)$ is surjective, because tensor product is right exact.

We see from this diagram that
\[
\cH^0(M_s \xrightarrow{\alpha_s} N_s \xrightarrow{\beta_s} K_s) \cong ({\rm ker}(\beta) \otimes_R R_s) / ({\rm im}(\alpha) \otimes_R R_s).
\]
Because ${\rm ker}(\beta)/{\rm im}(\alpha)$ is free over $A$, the short exact sequence:
\[
0 \to  {\rm im}(\alpha) \to {\rm ker}(\beta) \to {\rm ker}(\beta)/{\rm im}(\alpha) \to 0
\]
stays exact after tensoring by $k(s)$ over $A$, and so it stays exact after applying $(-) \otimes_R R_s$. Hence, we get the following identity as required:
\begin{align*}
\cH^0(M \xrightarrow{\alpha} N \xrightarrow{\beta} K) \otimes_R R_s &= \left({\rm ker}(\beta)/{\rm im}(\alpha)\right) \otimes_R R_s \\
&\cong ({\rm ker}(\beta) \otimes_R R_s) / ({\rm im}(\alpha) \otimes_R R_s). \qedhere
\end{align*}
\end{proof}

Let $A \to B$ be a map of rings and let $M, N$ be $A$-modules. Then there exist natural maps:
\[
{\rm Hom}_A(M,N) \otimes_A B \to {\rm Hom}_A(M,N \otimes_A B) \xrightarrow{\cong} {\rm Hom}_B(M \otimes_A B, N \otimes_A B).
\]
In turn, these give us natural maps
\[
{\rm Ext}^i_A(M,N) \otimes_A B \to {\rm Ext}^i_A(M,N \otimes_A B) \xrightarrow{\cong} {\rm Ext}^i_B(M \otimes_A B, N \otimes_A B)
\]
for every $i$.
\begin{proposition}\label{prop:reduction}
Let $A$ be a Noetherian domain and let $R$ be a finitely generated $A$-algebra. Let $Q$ be a finite $R$-module and let $M^\bullet \in D^{b}_{\rm fg}(R)$ be an element of the bounded derived category of finite $R$-modules. 

Fix an integer $i$. Then there exists an open subset $U \subseteq \Spec A$ such that for every closed point $s \in \Spec A$, the natural map
\[
\Ext^i_{R}(M^\bullet, Q) \otimes_R R_s \to \Ext^i_{R_s}(M^\bullet_s, Q_s)
\]
is an isomorphism. 
\end{proposition}

\noindent \noindent Here $R_s \coloneqq R \otimes_A k(s)$, $Q_s \coloneqq Q \otimes_R R_s$, and $M^\bullet_s \coloneqq M \otimes^L_R R_s$.

\begin{proof}
Up to quasi-isomorphism, we can replace $M^\bullet$ by a bounded from above complex of free $R$-modules:
\[
\ldots \to R^{\oplus a_{i+1}} \xrightarrow{A_{i+1}} R^{\oplus a_{i}} \xrightarrow{A_{i}} R^{\oplus a_{i-1}} \to \ldots
\]
where $A_i$ are matrices with coefficients in $R$. Then $\Ext^i_{R}(M^\bullet, Q)$ is calculated by applying ${\rm Hom}_R(-,Q)$ to the above complex and taking the $i$-th cohomology. 

In turn, by definition, $M \otimes^L_A k(s)$ is calculated by applying $(-)\otimes_R R_s$ to the above complex so that we get
\[
\ldots \to R^{\oplus a_{i+1}}_s \xrightarrow{A_{i+1,s}} R^{\oplus a_{i}}_s \xrightarrow{A_{i,s}} R^{\oplus a_{i-1}}_s \to \ldots
\]
Then $\Ext^i_{R_s}(M_s^\bullet, Q_s)$ is calculated by applying ${\rm Hom}_{R_s}(-,Q_s)$ to the above complex and taking the $i$-th cohomology.\\

With the above notation, our goal is to show that the cohomology of
\[
\underbrace{{\rm Hom}_R(R^{\oplus a_{i+1}},Q)}_{Q^{\oplus a_{i+1}}} \leftarrow \underbrace{{\rm Hom}_R(R^{\oplus a_{i}},Q)}_{Q^{\oplus a_{i}}} \leftarrow \underbrace{{\rm Hom}_R(R^{\oplus a_{i-1}},Q)}_{Q^{\oplus a_{i-1}}}
\]
tensored by $R_s$ over $R$, is isomorphic to the cohomology of
\[
\underbrace{{\rm Hom}_{R_s}(R_s^{\oplus a_{i+1}},Q_s)}_{Q^{\oplus a_{i+1}} \otimes_R R_s} \leftarrow \underbrace{{\rm Hom}_{R_s}(R_s^{\oplus a_{i}},Q_s)}_{Q^{\oplus a_{i+1}}\otimes_R R_s} \leftarrow \underbrace{{\rm Hom}_{R_s}(R_s^{\oplus a_{i-1}},N_s)}_{Q^{\oplus a_{i+1}}\otimes_R R_s}
\]
for closed points $s$ in some open subset $U \subseteq \Spec A$. This follows immediately from Lemma \ref{lem:reductionmodp-basic}.  \qedhere
\end{proof}

\begin{proposition}\label{prop:reduction(local-coh)}
     Let $X$ be a normal affine variety over a field $K$ of characteristic zero, and let $f \colon M \to N$ be a map between $M, N \in D^b_{\rm coh}(X)$. Let $X_A$ be flat model of $X$ over a finitely generated $\bZ$-subalgebra $A \subseteq K$. Suppose there exist $M_A, N_A \in D^b_{\rm coh}(X_A)$ and $f_A \colon M_A \to N_A$ such that 
     \[
     M_A \otimes^L_A K \cong M \quad N_A \otimes^L_A K \cong M \quad f_A \otimes^L_A K = f.
     \]
     Fix an integer $j \geq 0$. Then there exists an open subset $U \subseteq \Spec A$ such that the following statements are equivalent:
     \begin{enumerate}
         \item the map
     \[
     f^* \colon H^j_\m(M) \to H^j_\m(N)
     \]
     is injective (resp.\ surjective) for every closed point $\m \in X$,
     \item for some (equivalently, every) closed point $s \in U$ the map
     \[
     f_s^* \colon H^j_\n(M_s) \to H^j_\n(N_s)
     \]
     is injective (resp.\ surjective) for every closed point $\n \in X_s$.
     \end{enumerate} 
\end{proposition}

\begin{proof}
Consider a closed embedding $u \colon X \to W$ to a smooth over $K$ affine variety $W \coloneqq \Spec K[x_1,\ldots, x_m]$. We may assume that $u$ extends to a closed embedding of models $u_A \colon X_A \to W_A$ where $W_A \coloneqq \Spec A[x_1,\ldots, x_m]$ is smooth over $\Spec A$. Let $u_s \colon X_s \to W_s$ be the restriction of $u$ to the fibre over $s \in \Spec A$. 
Since
\begin{align*}
H^j_\m(M) &= H^j_\m(Ru_*M) \quad &&\text{ and }&& H^j_\m(N) = H^j_\m(Ru_*N)  \\
H^j_\n(M_s) &= H^j_\n(R(u_s)_*M_s) \quad &&\text{ and }&& H^j_\n(N_s) = H^j_\n(R(u_s)_*N_s)
\end{align*}
we can replace $X_A$ by $W_A$, and assume from now on that $X_A$ is smooth over $\Spec A$. 

Now, let $d$ be the dimension of $X$ over $K$. Write $X = \Spec R$,  $X_A = \Spec R_A$, and identify $M_A$ and $N_A$ with elements of the derived category $D^b_{\rm fg}(R_A)$ of finitely generated $R_A$-modules.  By Matlis duality on $X$, the problem reduces to showing that there exists an open subset $U \subseteq \Spec A$ such that
\[
\Ext^{d-j}_{R_A}(M_A, \Omega^d_{R_A/A}) \xleftarrow{(\star)}\Ext^{d-j}_{R_A}(N_A, \Omega^d_{R_A/A})
\]
is injective (resp.\ surjective) over $K$ if and only if
\[
\Ext^{d-j}_{R_s}(M_s, \omega_{R_s}) \xleftarrow{(\star\star)} \Ext^{d-j}_{R_s}(N_s, \omega_{R_s})
\] 
is injective (resp.\ surjective) for some (equivalently, every) closed point $s \in U$. We warn the reader that $\Omega^d_{R_A/A} \not \cong \omega_{R_A}$ as the absolute dimension of $R_A$ is $d+1$.\\ 

Consider the following diagram:
\begin{equation} \label{eq:BCdef}
0 \to B \to \Ext^{d-j}_{R_A}(N_A, \Omega^d_{R_A/A}) \xrightarrow{(\star)} \Ext^{d-j}_{R_A}(M_A, \Omega^d_{R_A/A}) \to C \to 0,
\end{equation}
where $B$ and $C$ are the kernel and cokernel of the map $(\star)$, respectively. By generic freeness \cite[Tag 051S]{stacks-project}, we can replace $\Spec A$ by an open subset so that both 
\begin{equation} \label{eq:BCfree}
\text{$B$ and $C$ are free over $A$.}
\end{equation}
Further, we replace $\Spec A$ by another open subset using Proposition \ref{prop:reduction} so that
\begin{align*}
\Ext^{d-j}_{R_A}(N_A, \Omega^d_{R_A/A}) \otimes_R R_s &\cong \Ext^{d-j}_{R_s}(N_s, \omega_{R_s}), \text{ and } \\
\Ext^{d-j}_{R_A}(M_A, \Omega^d_{R_A/A}) \otimes_R R_s &\cong \Ext^{d-j}_{R_s}(M_s, \omega_{R_s})
\end{align*}
for every closed point $s \in \Spec A$, where $\Omega^d_{R_A/A} \otimes_R R_s \cong \Omega^d_{R_s} = \omega_{R_s}$ as $R_s$ is smooth. In turn, we can replace $\Spec A$ by yet another open subset using Lemma \ref{lem:reductionmodp-basic} so that
\begin{equation} \label{eq:sesReducedBC}
0 \to B_s \to \Ext^{d-j}_{R_s}(N_s, \omega_{R_s}) \xrightarrow{(\star\star)} \Ext^{d-j}_{R_s}(M_s, \omega_{R_s}) \to C_s \to 0
\end{equation}
is exact for every $s \in \Spec A$, where $B_s \coloneqq B \otimes_R R_s$ and $C_s \coloneqq C \otimes_R R_s$.

Then we have the following equivalences:
\begin{align*}
\text{$(\star)$ is injective over $K$} &\overset{(\ref{eq:BCdef})}{\iff} B \otimes_A K = 0 \\
&\overset{(\ref{eq:BCfree})}{\iff}B = 0 \\
&\overset{(\ref{eq:BCfree})}{\iff} B_s = 0 \text{ for some/every closed point $s \in \Spec A$} \\
&\overset{(\ref{eq:sesReducedBC})}{\iff} \text{$(\star\star)$ is injective  for some/every closed point $s \in \Spec A$}
\end{align*}
\begin{align*}
\text{$(\star)$ is surjective over $K$} 
&\overset{(\ref{eq:BCdef})}{\iff} C \otimes_A K = 0 \\
&\overset{(\ref{eq:BCfree})}{\iff} C = 0 \\
&\overset{(\ref{eq:BCfree})}{\iff} C_s = 0 \text{ for some/every closed point $s \in \Spec A$} \\
&\overset{(\ref{eq:sesReducedBC})}{\iff} \text{$(\star\star)$ is surjective  for some/every closed point $s \in \Spec A$}.
\end{align*}
This concludes the proof of the proposition.
\end{proof}

\begin{remark} \label{rem:reduction-zero}
The same argument works to  detect that $f^* \colon H^j_\m(M) \to H^j_\m(N)$ is zero if and only if  $f_s^* \colon H^j_\n(M_s) \to H^j_\n(N_s)$ is zero.
\end{remark}

\begin{corollary}\label{cor:reduction(local-coh)-Omega}
    Let $X$ be a normal affine variety over a field $K$ of characteristic zero, and let $\pi\colon Y\to X$ be a log resolution with reduced exceptional divisor $E$. Let $\pi_A \colon Y_A \to X_A$, with exceptional divisor $E_A$, be a model of $\pi$ over a finitely generated $\bZ$-subalgebra $A \subseteq K$. Fix integers $i,j \geq 0$.
    
    Then there exists an open subset $U \subseteq \Spec A$ such that the following are equivalent:

    \begin{enumerate}
         \item the map
     \[
      H^j_\m(\pi_{*}\Omega^i_Y(\log E)) \to H^j_\m(R\pi_{*}\Omega^i_Y(\log E))
     \]
     is injective (resp.\ surjective) for every closed point $\m \in X$,
     \item for some (equivalently, every) closed point $s \in U$ the map
     \[
     H^j_\n(\pi_{s,*}\Omega^i_{Y_s}(\log E_s)) \to H^j_\n(R\pi_{s,*}\Omega^i_{Y_s}(\log E_s))
     \]
     is injective (resp.\ surjective) for every closed point $\n \in X_s$.
     \end{enumerate} 
     Moreover, $\pi_{*}\Omega^{i}_Y(\log E) \to \Omega^{[i]}_X$ is an isomorphism if and only if $\pi_{s,*}\Omega^{i}_{Y_s}(\log {E_s}) \to \Omega^{[i]}_{X_s}$ 
is an isomorphism for some (equivalently, every) closed point $s \in U$.
\end{corollary}
\noindent Here $Y_s$, $E_s$, and $\pi_s$ are restrictions of $Y_A$, $E_A$, and $\pi_A$ over $k(s)$. 
\begin{proof}
By replacing $\Spec A$ by an affine open subset $U$ we may assume that $\pi_s \colon Y_s \to X_s$ is a log resolution of singularities of $X_s$ with simple normal crossing divisor $E_s$. Moreover, by shrinking $\Spec A$ even further, we may assume by generic freeness \cite[\href{https://stacks.math.columbia.edu/tag/051S}{Tag 051S}]{stacks-project} and Lemma \ref{lem:Hara's lemma} that
\begin{align*}
\pi_{A,*}\Omega^i_{Y_A/A}(\log E_A) \otimes_{\cO_{X_A}} \cO_{X_s} &\cong \pi_{s,*}\Omega^i_{Y_s}(\log E_s), \text{ and } \\[0.3em]
R\pi_{A,*}\Omega^i_{Y_A/A}(\log E_A) \otimes^L_{\cO_{X_A}} \cO_{X_s} &\cong R\pi_{s,*}\Omega^i_{Y_s}(\log E_s).
\end{align*}
Then the statement of the corollary follows immediately from Proposition \ref{prop:reduction(local-coh)}.

The \emph{moreover} statement follows from the claim in the proof of \cite[Corollary 3.6]{Kawakami-Sato}. Alternatively the same argument as above shows that  
\[
H^j_\m(\pi_{*}\Omega^i_Y(\log E))=0 \quad \text{ for all closed points $\m \in X$}
\]
if and only if for some (equivalently, every) closed point $s \in U$:
\[
H^j_\n(\pi_{s,*}\Omega^i_{Y_s}(\log E_s))=0 \quad \text{ for all closed points $\n \in X_s$.}
\]
Then the assertion can be deduced from Lemma \ref{lem:reflexive}. \qedhere
\end{proof}

\section{Definitions and foundational properties of higher $F$-rationality}

First, we introduce the notion of pre-$m$-$F$-rationality.
Let $X$ be a normal variety {over a perfect field $k$ of characteristic $p>0$}.
Suppose that $C \colon Z\Omega^{[i]}_{X} \to \Omega^{[i]}_{X}$ is surjective.
Then we have
the iterated inverse Cartier operator \eqref{eq:inverse reflexive Cartier op}
\[
    C^{-1}_n\colon \Omega^{[i]}_X \underset{\cong}{\xleftarrow{\,C_{n}\,}}  \frac{Z_n\Omega^{[i]}_{X}}{B_n\Omega^{[i]}_{X}} \xhookrightarrow{\hphantom{ C_n}} \frac{F^n_{*}\Omega^{[i]}_{X}}{B_n\Omega^{[i]}_{X}}.   
\]

\begin{definition} \label{def:prekFinj-iso}
Let $X$ be a $d$-dimensional normal affine variety over a perfect field $k$ of characteristic $p>0$.
Fix an integer $m \geq 0$. 
We say $X$ is \emph{pre-$m$-$F$-rational} if
\begin{enumerate}
\item $C \colon Z\Omega^{[i]}_{X} \to \Omega^{[i]}_{X}$ is surjective for all $i\leq m$, and\\[-0.9em]
\item for every closed point $\m \in X$ and every Cartier divisor $D\geq 0$ on $X$, there exists a positive integer $n_0>0$ such that the composition
\[
H^j_\m(\Omega^{[i]}_{X}) \xrightarrow{C_n^{-1}} H^j_\m\left(\frac{F^n_{*}\Omega^{[i]}_{X}}{B_n\Omega^{[i]}_{X}}\right) \xrightarrow{\mathrm{nat.}} H^j_\m\left(\frac{F^n_{*}\Omega^{[i]}_{X}(D)}{B_n\Omega^{[i]}_{X}}\right)
\]
is injective for all $i,j \geq 0$ and $n \geq n_0$ such that $i\leq m$ and $i+j \leq d$.
\end{enumerate}
We say that $X$ is \emph{$m$-$F$-rational} if it is pre-$m$-$F$-rational and $\mathrm{codim}_{X}\, {\rm Sing}(X) > 2m+1$.
\end{definition}

Note that if $X$ is pre-$m$-$F$-rational, then
$\pi_{*}\Omega^{i}_Y(\log E)\cong \Omega^{[i]}_X$ holds for all $i\leq m$ by Condition (1) and Theorem \ref{thm:LET}, where $\pi\colon Y\to X$ is a log resolution with reduced exceptional divisor $E$.\\

Next, we prove Theorem \ref{thm:IndependenceIntro} treating equivalent definitions of higher $F$-rationality.

\begin{proof}[{Proof of Theorem \ref{thm:IndependenceIntro}}]
Clearly (1) $\implies$ (2). We start by discussing the implication (2) $\implies$ (3). We need to find an integer $n_0>0$ such that the composition
\begin{equation} \label{eq:goalmapequivalent}
H^j_\m(\Omega^{[i]}_{X}) \xrightarrow{C_n^{-1}} H^j_\m\left(\frac{F^n_{*}\Omega^{[i]}_{X}}{B_n\Omega^{[i]}_{X}}\right) \xrightarrow{\delta} H^{j+1}_\m(B_n\Omega^{[i]}_X) 
\end{equation}
is injective when $0 < i \leq m$, $i+j \leq d$, and $n \geq n_0$. 

Pick an effective divisor $D'$ containing ${\rm Sing}(X)$. We have  $j< d$, and so by Matlis duality we can pick $r>0$ such that the natural map
\[
H^j_\m(\Omega^{[i]}_{X}) \to H^j_\m(\Omega^{[i]}_{X}(D))
\]
is a zero map where $D \coloneqq rD'$. Here we are using that the Matlis dual of $H^j_\m(\Omega^{[i]}_{X})$ for $j<d$ is supported on ${\rm Sing}(X)$.

We may assume that (2) holds for $D$. Now consider the following diagram
\begin{equation} \label{eq:independentBigDiagram}
\begin{tikzcd}[column sep = large, row sep = large]
&  H^j_\m(F^n_{*}\Omega^{[i]}_{X}) \ar{r}{0} \ar{d} & H^j_\m(F^n_{*}\Omega^{[i]}_{X}(D)) \ar{d} \\
H^j_\m(\Omega^{[i]}_{X}) \ar{r}{C_n^{-1}} & H^j_\m\left(\frac{F^n_{*}\Omega^{[i]}_{X}}{B_n\Omega^{[i]}_{X}}\right) \ar{d}{\delta} \ar{r}{\rm nat} & H^j_\m\left(\frac{F^n_{*}\Omega^{[i]}_{X}(D)}{B_n\Omega^{[i]}_{X}}\right) \ar{d}{\delta_D} \\
& H^{j+1}_\m(B_n\Omega^{[i]}_X) \ar{r}{=} & H^{j+1}_\m(B_n\Omega^{[i]}_X).
\end{tikzcd}
\end{equation}
Here $\delta_D$ is induced by the long exact sequence of local cohomology associated to:
\[
0 \to B_n\Omega^{[i]}_{X} \to F^n_{*}\Omega^{[i]}_{X}(D) \to \frac{F^n_{*}\Omega^{[i]}_{X}(D)}{B_n\Omega^{[i]}_{X}} \to 0.
\]
We see from the above diagram that (\ref{eq:goalmapequivalent}) is injective when the composition 
\begin{equation} \label{eq:independentDdef}
H^j_\m(\Omega^{[i]}_{X}) \xrightarrow{C_n^{-1}} H^j_\m\left(\frac{F^n_{*}\Omega^{[i]}_{X}}{B_n\Omega^{[i]}_{X}}\right) \xrightarrow{\rm nat} H^j_\m\left(\frac{F^n_{*}\Omega^{[i]}_{X}(D)}{B_n\Omega^{[i]}_{X}}\right)
\end{equation}
is injective which is our assumption. This concludes the proof of (2) $\implies$ (3).

In order to show that (3) $\implies$ (1), we pick any effective Cartier divisor $D$. Up to enlarging $D$, we may assume as before that
\[
H^j_\m(\Omega^{[i]}_{X}) \to H^j_\m(\Omega^{[i]}_{X}(D))
\]
is a zero map for $j<d$. By tracing through the above diagram \eqref{eq:independentBigDiagram}, we get that the injectivity of \eqref{eq:goalmapequivalent} implies the injectivity of \eqref{eq:independentDdef} for $i>0$ as desired. The case of $i=0$ follows from the assumption on classical $F$-rationality.

Finally, to show the equivalence (3) $\iff$ (4), we first consider the following diagram
\[
\begin{tikzcd}
0 \ar{r} & B_n\Omega^{[i]}_X \ar{r} & F^n_*\Omega^{[i]}_X \ar{r} & \frac{F^n_{*}\Omega^{[i]}_{X}}{B_n\Omega^{[i]}_{X}} \ar{r} & 0\\  
0 \ar{r} & B_n\Omega^{[i]}_X \ar{u}{=} \ar{r} & Z_n\Omega^{[i]}_X \ar{u} \ar{r}{C_n} & \Omega^{[i]}_X \ar{u}[swap]{C^{-1}_n} \ar{r} & 0    
\end{tikzcd}
\]
with both rows exact. The construction of such a diagram  follows from the definition of higher cycles and boundaries. In turn, we get the induced map of long exact sequences of local cohomology
\[
\begin{tikzcd}
H^j_\m(F^n_*\Omega^{[i]}_X) \ar{r} & H^j_\m\left(\frac{F^n_{*}\Omega^{[i]}_{X}}{B_n\Omega^{[i]}_{X}}\right) \ar{r}{\delta} & H^{j+1}_\m(B_n\Omega^{[i]}_X)  \\
H^j_\m(Z_n\Omega^{[i]}_X) \ar{u} \ar{r}{C_n} & H^j_\m(\Omega^{[i]}_X) \ar{u}[swap]{C^{-1}_n} \ar{r}{\delta'} & H^{j+1}_\m(B_n\Omega^{[i]}_X). \ar{u}[swap]{=}
\end{tikzcd}
\]
In particular, in (3) we can replace the injectivity of $\delta \circ C^{-1}_n$ by the injectivity of $\delta'$, which is the same as the map $C_n$ being zero. Thus (3) $\iff$ (4) as required.
\end{proof}

\begin{corollary} \label{cor:correctDef}
 Let $X$ be a $d$-dimensional normal affine variety over a perfect field $k$ of characteristic $p>0$. Then $X$ is $m$-$F$-rational if and only if 
 \begin{enumerate}
     \item $\mathrm{codim}_{X}\, {\rm Sing}(X) > 2m+1$, \\[-0.8em]
     \item $X$ is $F$-rational, and \\[-0.8em] 
     \item there exists an integer $n_0>0$ such that the composition
\[
H^j_\m(\Omega^{[i]}_{X}) \xrightarrow{C_n^{-1}} H^j_\m\left(\frac{F^n_{*}\Omega^{[i]}_{X}}{B_n\Omega^{[i]}_{X}}\right) \xrightarrow{\delta} H^{j+1}_\m(B_n\Omega^{[i]}_X) 
\]
is injective for all $i, j, n$ such that $0 < i \leq m$,  $i+j \leq d$, and $n \geq n_0$.
 \end{enumerate}
\end{corollary}
\begin{proof}
This is immediate by definition and Theorem \ref{thm:IndependenceIntro}.
\end{proof}

\subsection{Vanishing of local cohomologies}

In this subsection, we obtain results on the vanishing of local cohomologies $H^j_{\m}(\Omega^i_{X})$ for a pre-$m$-$F$-rational variety $X$. 

First, we verify the following proposition.

\begin{proposition}\label{thm:mCMBvanishing}
Let $X$ be a normal $d$-dimensional {affine} variety defined over a perfect field of characteristic $p>0$. Fix an integer $0 \leq m \leq \codim_{X} {\rm Sing}(X) - 3$ and a closed point $\m \in {\rm Sing}(X)$. Further, assume that for all $0 \leq i \leq m$:
\begin{enumerate}
\item $C \colon Z\Omega^{[i]}_{X} \to \Omega^{[i]}_{X}$ is surjective, \\[-0.9em]
\item the following map is injective  when $i + j \leq d$: \[H^j_\m(\Omega^{[i]}_{X}) \xrightarrow{C^{-1}} H^j_\m\left(\frac{F_{*}\Omega^{[i]}_{X}}{B\Omega^{[i]}_{X}}\right),
\] 
\item $H^j_\m(\Omega^{[i]}_{X}) = 0$ when $i+j<d$.
\end{enumerate}
Then the following vanishing holds
\begin{equation} \label{eq:BVan}
H^j_{\m}(B_n\Omega^{[i]}_{X})=0
\end{equation}
 for all $0 \leq i \leq m+1$ and all $n\geq 1$ when $i+j \leq d$.
\end{proposition}
\noindent 
Conditions (1) and (2) in the proposition constitute a reflexive variant of the definition of $m$-$F$-injectivity. Condition (3) may be thought of as higher variant of Cohen--Macaulayness.
\begin{proof} We start by showing the proposition for $n=1$. To this end, by ascending induction on $m$ we may assume that
\begin{equation} \label{eq:BVanAssump}
H^j_{\m}(B\Omega^{[m]}_{X})=0
\end{equation}
when $j \leq d-m$ and aim for showing that
\begin{equation} \label{eq:BVanGoal}
H^j_{\m}(B\Omega^{[m+1]}_{X})=0
\end{equation}
when $j \leq d-m-1$. First, by the long exact sequence of local cohomology for
\[
0\to B\Omega^{[m]}_X\to F_{*}\Omega^{[m]}_X \to \frac{F_{*}\Omega^{[m]}_X}{B\Omega^{[m]}_X} \to 0
\]
together with (3) and (\ref{eq:BVanAssump}) we get that 
\begin{equation} \label{eq:GVan}
H^j_{\m}\left(\frac{F_{*}\Omega^{[m]}_X}{B\Omega^{[m]}_X}\right)=0
\end{equation}
for all $j < d-m$. Next, consider the short exact sequence
\[
0\to \Omega^{[m]}_X\xrightarrow{C^{-1}} \frac{F_{*}\Omega^{[m]}_X}{B\Omega^{[m]}_X} \to \mathcal{C}\to 0,
\]
where $\mathcal{C}$ denotes the cokernel of the first map. By (2) and (\ref{eq:GVan}), we get that
\begin{equation} \label{eq:Cvan}
H^j_{\m}(\mathcal{C})=0
\end{equation}
for all $j< d-m$. Since 
\[
d-m \geq {\rm dim}_X\, {\rm Sing}(X) + 3,
\]
Remark \ref{rem:Sk} shows that $\cC$ satisfies ($S_3$). In particular $\mathcal{C}=B\Omega^{[m+1]}_X$. Consequently, \eqref{eq:Cvan} yields (\ref{eq:BVanGoal}) as required. \\

Next, we verify (\ref{eq:BVan}) for every $n \geq 1$. To this end, consider the short exact sequence
\[
0\to F^{n-1}_{*}B\Omega^{[m+1]}_X\to B_{n}\Omega^{[m+1]}_X \xrightarrow{C} B_{n-1}\Omega^{[m+1]}_X,
\]
which is a reflexivisation of \eqref{eq:BBB}. Since $B\Omega^{[{m+1}]}_X$ satisfies $(S_3)$ as proven above, taking into account that $\mathrm{codim}\,\mathrm{Sing}(X) 
\geq 3$ and so $C\colon B_{n}\Omega^{[m+1]}_X \to B_{n-1}\Omega^{[m+1]}_X$ is surjective at codimension two points, the above  sequence is exact on the right.
Arguing by ascending induction on $n$, we see that the vanishing
\[
H^j_{\m}(B_{n-1}\Omega^{[{m+1}]}_X)=0
\]
for all $j<d-m$ implies the vanishing $H^j_{\m}(B_{n}\Omega^{[{m+1}]}_X)=0$ for all $j < d - m$ as well. This concludes the proof of \eqref{eq:BVan}. \qedhere
\end{proof}

\begin{theorem}\label{thm:vanishing of local cohomologies} 
    Let $X$ be a normal $d$-dimensional {affine} variety defined over a perfect field of characteristic $p>0$ with pre-$m$-$F$-rational singularities for some integer $0 \leq m \leq \codim_{X} {\rm Sing}(X) - 2$.
    Then 
    \begin{equation} \label{eq:vanishing-goal}
    H^j_{\m}(\Omega^{[i]}_{X})=0
    \end{equation}
    for all maximal ideals $\m \in {\rm Sing}(X)$ when $i\leq m$ and $i+j < d$.
\end{theorem}
\begin{proof}
Since $X$ is $F$-rational, it is Cohen--Macaulay. By ascending induction on $m$ we may assume that (\ref{eq:vanishing-goal}) holds when $i \leq m-1$ and $i+j < d$. Since pre-$(m-1)$-$F$-rationality implies (1) and (2) in Proposition \ref{thm:mCMBvanishing} (applied with $m$ replaced by $m-1$), we get that
\begin{equation} \label{eq:anotherBvan}
H^j_{\m}(B_n\Omega^{[m]}_{X})=0
\end{equation}
for all $n\geq 1$ when $j \leq d - m$.

On the other hand, Definition (3) of pre-$m$-$F$-rationality from Theorem \ref{thm:IndependenceIntro} stipulates that the composition
\[
H^j_\m(\Omega^{[m]}_{X}) \xrightarrow{C_n^{-1}} H^j_\m\left(\frac{F^n_{*}\Omega^{[m]}_{X}}{B_n\Omega^{[m]}_{X}}\right) \xrightarrow{\delta} H^{j+1}_\m(B_n\Omega^{[m]}_X) 
\]
is injective when $j \leq d-m$ and $n \gg 0$. This combined with \eqref{eq:anotherBvan} concludes the vanishing \eqref{eq:vanishing-goal} for $i=m$ and $i+j<d$. \qedhere
\end{proof}

\begin{remark} \label{remark:ZZeroDef}
The above results lead to another definition of $m$-$F$-rationality. Namely, let $X$ be a $d$-dimensional normal affine variety over a perfect field $k$ of characteristic $p>0$. Then $X$ is $m$-$F$-rational if and only if 
 \begin{enumerate}
     \item $\mathrm{codim}_{X}\, {\rm Sing}(X) > 2m+1$, \\[-0.8em]
     \item $X$ is $F$-rational, and \\[-0.8em] 
     \item there exists an integer $n_0>0$ such that
\[
H^j_\m(Z_n\Omega^{[i]}_{X}) = 0
\]
 for all $i, j, n$ satisfying $0 < i \leq m$,  $i+j \leq d$, and $n \geq n_0$.
 \end{enumerate}
This can be easily deduced from Theorem \ref{thm:IndependenceIntro} (4), Theorem \ref{thm:vanishing of local cohomologies}, and Proposition \ref{thm:mCMBvanishing} by way of the short exact sequence
\[
0 \to B_n\Omega^{[i]}_X \to Z_n\Omega^{[i]}_X \to \Omega^{[i]}_X \to 0.
\]
\end{remark}

Next, we give a criterion for checking whether a singularity is $m$-$F$-rational that is often easy to apply in practice. Let us point out that the last assumption is very restrictive and most $m$-$F$-rational singularities do not satisfy it.
\begin{proposition}\label{prop: prop for examples}
    Let $X$ be a normal $d$-dimensional affine variety defined over a perfect field of characteristic $p>0$. 
    Suppose that
\begin{enumerate}
\item $X$ is $F$-rational, 
\item  $C \colon Z\Omega^{[i]}_{X} \to \Omega^{[i]}_{X}$ is surjective, \\[-0.9em]
\item $H^j_\m(\Omega^{[i]}_{X}) = 0$ for all maximal ideal $\m$ when $0 <  i\leq m$ and $i+j\leq d$.
\end{enumerate}
Then $X$ is pre-$m$-$F$-rational.
\end{proposition}
\begin{proof}
    The assertion follows from Theorem \ref{thm:IndependenceIntro} (3)$\Rightarrow$(1).
\end{proof}

Thanks to Theorem \ref{thm:pre-k-F-rational type to pre-$k$-rational}, we can construct plenty of examples of $m$-$F$-rational singularities by starting with an $m$-rational singularity in characteristic $0$ and reducing modulo $p \gg 0$. Below we give some explicit examples that work in low characteristics as well.

\begin{example} \label{example1}
    Let $X$ be a normal surface with $F$-rational and $F$-pure singularities defined over a perfect field of characteristic $p>0$. We claim that $X$ is pre-$m$-$F$-rational for every $m \geq 0$.  First, we note that $X$ has rational singularities since $F$-rationality implies pseudo-rationality and the Grauert–Riemenschneider vanishing holds for surfaces. 
    Then $X$ is $F$-pure and rational, and so
    \[
    C \colon Z\Omega^{[1]}_X\to \Omega^{[1]}_X
    \]
     is surjective by  \cite[Theorem B]{Kawakami-Takamatsu} and \cite[Proposition 4.4]{Kaw4}. Moreover, the trace map of Frobenius $F_{*}\omega_X\to \omega_X$  is surjective by $F$-purity. Then the claim follows from the above proposition given that 
    \[
    H^j_\m(\Omega^{[i]}_{X})=0
    \]
    for $(i,j) \in \{ (1,0), (1,1), (2,0)\}$ as $\Omega^{[i]}_{X}$ is reflexive by definition.
\end{example}

\begin{example} \label{example2}
    Let $X$ be a simplicial toric variety of dimension $d$ defined over a perfect field of characteristic $p>0$. We claim that $X$ is pre-$m$-rational for all $m\geq 0$.
    To this end, since $X$ is $F$-rational, we may assume that $m\geq 1$.
    Since $X$ is simplicial, we have an exact sequence
    \[
    0\to \Omega^{[m]}_X\to K^0\to K^1\to \cdots \to K^m \to 0,
    \]
    where $K^i$ is a direct sum of structure sheaves of toric varieties of dim $d-i$ (see \cite[Theorem 8.2.19]{CLS11}). Since toric varieties are Cohen--Macaulay, we can deduce from this exact sequence that 
    \[
    H^j_{\m}(\Omega^{[m]}_X)=0
    \]
    for all $j < d$. 
    Since $X$ is Frobenius liftable, 
    $C \colon Z\Omega^{[m]}_{X} \to \Omega^{[m]}_{X}$ is surjective (see \cite[Remark 2.7 (3) and Lemma 3.8]{Kaw4}).
    Thus, by Proposition \ref{prop: prop for examples}, we conclude that $X$ is pre-$m$-$F$-rational.
\end{example}

\begin{example} \label{example3} 
     Let $x \in X$ be a linearly reductive quotient singularity over an algebraically closed field of characteristic $p>0$,
    that is, there exists a faithful action of a finite group scheme $G$ on $\Spec k\llbracket x_1,\ldots,x_d\rrbracket$ fixing the closed point such that $\sO_{X,x}^{\wedge}\cong k\llbracket x_1,\ldots,x_d\rrbracket^{G}$.
    We claim that $x\in X$ is pre-$m$-rational for all $m\geq 0$.
    To this end, since $x\in X$ is Frobenius liftable \cite[Theorem 2.12]{Kaw4}, it suffices to show that
    \[
    H^j_{\m_x}(\Omega^{[i]}_X)=0
    \]
    for all $i+j\leq d$.
    By Artin approximation \cite[Corollary 2.6]{Artin(approximation)}, there exist \'etale morphisms $\pi\colon Z\to X$ and $\pi'\colon Z\to \mathbb{A}^d_k/G$ and a closed point $z\in Z$ such that $\pi(z)=x$ and $\pi'(z)=o$, where $o\in \mathbb{A}^d_k/G$ is the origin.
    Then it suffices to show that 
    \[
    H^j_{\m_o}(\Omega^{[i]}_{\mathbb{A}_{k}^{d}/G})=0
    \]
    for all $i+j\leq d$.
    Consider an isomorphism $\mathbb{A}_{k}^{d}/G\cong (\mathbb{A}_{k}^{d}/G^{\circ})/(G/G^{\circ})$, where $G^{\circ}$ is the connected component containing the identity. Then $\mathbb{A}_{k}^{d}/G^{\circ}$ is a normal toric variety and $G/G^{\circ}$ is a finite group scheme order prime to $p$ (see \cite[Theorem 2.12]{Kaw4} for more details).
    Since $\mathbb{A}_{k}^{d}\to \mathbb{A}_{k}^{d}/G^{\circ}$ is finite, the quotient $\mathbb{A}_{k}^{d}/G^{\circ}$ is $\Q$-factorial.
    Then $\Omega^{[i]}_{\mathbb{A}_{k}^{d}/G}$ is a direct summand of $\Omega^{[i]}_{\mathbb{A}_{k}^{d}/G^{\circ}}$, which satisfies the desired vanishing, as proven in the previous example.
\end{example}

Next, we generalise \cite[Theorem D]{Kaw4} to the setting of $m$-$F$-rationality.

\begin{theorem} \label{thm:higherRationalGivesExtension}
Let $X$ be a normal variety over a perfect field of characteristic $p>0$ with pre-$m$-$F$-rational singularities. 
Suppose that $m \leq \codim_{X} {\rm Sing}(X)-3$.
Then the natural restriction map
\[
f_{*}\Omega^{[m+1]}_Y(\log E)\hookrightarrow\Omega^{[m+1]}_X
\]
is surjective for every proper birational morphism $f\colon Y\to X$ with reduced exceptional divisor $E$.
\end{theorem}
\begin{proof}
    As proven in the above theorems (or by combining the statements of Proposition \ref{thm:mCMBvanishing}, Theorem \ref{thm:vanishing of local cohomologies}, and Remark \ref{rem:Sk}), $B\Omega^{[m+1]}_X$ satisfies $(S_3)$, and therefore
    \[
    0\to B\Omega^{[m+1]}_X \to Z\Omega^{[m+1]}_X \to \Omega^{[m+1]}_X
    \]
    is exact on the right. Then the assertion holds by Theorem \ref{thm:LET}.
\end{proof}
Note that in the above theorem we only need to assume the hypotheses of Proposition \ref{thm:mCMBvanishing} as opposed to full pre-$m$-$F$-rationality.

\section{$m$-$F$-rational and $m$-rational singularities}

\subsection{$m$-$F$-rational implies $m$-rational} \label{ss:towards-proof-1-iso}
In this subsection, we show that if a reduction of a characteristic $0$ singularity $X$ modulo $p\gg 0$ is pre-$m$-$F$-rational, then $X$ is pre-$m$-rational assuming that $m < {\rm codim}\, {\rm Sing}(X)$. 

We start by proving the following theorem.  

\begin{theorem}\label{thm:main-component-to-char-0}
Let $X$ be a normal  $d$-dimensional affine variety over a perfect field $k$ of characteristic $p>0$ with pre-$m$-$F$-rational singularities, and let $\pi\colon Y\to X$ be a log resolution with reduced exceptional divisor $E$ such that 
\[
R^{>0}\pi_{*}\Omega^i_Y(\log E)= 0
\]
for all $i\leq m-1$.
Then we have $\pi_{*}\Omega^{i}_Y(\log E)\cong \Omega^{[i]}_X$ for all $i\leq m$, and the map
\[
\pi^* \colon H^j_{\m}(\Omega^{[i]}_X)\to H^j_{\m}(R\pi_{*}\Omega^{i}_Y(\log E))
\]
is injective for every maximal ideal $\m$ and all integers $i, j$ when $i\leq m$ and $i + j \leq d$.
\end{theorem}
\begin{proof}
It is enough to verify the statement for $i=m$ and $j \leq d-m$. 
By Lemma \ref{thm:LET}(1), we have that $\pi_*\Omega^m_Y(\log E) = \Omega^{[m]}_X$. The rest of the theorem follows immediately from the existence and commutativity of the following diagram given Definition (3) of pre-$m$-$F$-rationality from Theorem \ref{thm:IndependenceIntro}:
\[
\begin{tikzcd}[column sep = large]
H^j_\m(\Omega^{[m]}_X) \ar{r}{C_n^{-1}} \ar{d}{\pi^*} & H^j_\m\left(\frac{F^n_*\Omega^{[m]}_X}{B_n\Omega^{[m]}_X}\right) \ar{d}{\pi^*} \ar{r}{\delta} & H^{j+1}_\m(B_n\Omega^{[m]}_X) \ar{d}{\pi^*}[swap]{=} \\
H^j_\m(R\pi_*\Omega^{m}_Y(\log E)) \ar{r}{C_n^{-1}} & H^j_\m\left(R\pi_*\frac{F^n_*\Omega^{m}_Y(\log E)}{B_n\Omega^{m}_Y(\log E)}\right) \ar{r}{R\pi_*\delta_Y} & H^{j+1}_\m(R\pi_*B_n\Omega^{m}_Y(\log E)).
\end{tikzcd}
\]

To construct this diagram, we first note that the existence of the left vertical map follows from $\pi_*\Omega^m_Y(\log E) = \Omega^{[m]}_X$. Similarly, Lemma \ref{thm:LET}(2) yields $\pi_*B_n\Omega^m_Y(\log E) = B_n\Omega^{[m]}_X$, which ensures that the right vertical arrow exists. Together, these results imply the existence of the middle vertical arrow. Finally, Lemma \ref{lemma:Kawakami-vanishing}(2) asserts that the rightmost vertical arrow is an equality.

The commutativity of the left square follows from the functoriality of $C_n^{-1}$. The commutativity of the right square follows from the long exact sequence of local cohomology associated to the following diagram:
\[
\begin{tikzcd}
    0 \ar{r} & B_n\Omega^{[m]}_X \ar{d}{\pi^*} \ar{r} & F^n_*\Omega^{[m]}_X \ar{d}{\pi^*} \ar{r} & \frac{F^n_*\Omega^{[m]}_X}{B_n\Omega^{[m]}_X} \ar{d}{\pi^*} \ar{r} & 0 \\
    0 \ar{r} & R\pi_*B_n\Omega^{m}_Y(\log E) \ar{r} & R\pi_*F^n_*\Omega^{m}_Y(\log E) \ar{r} & R\pi_*\frac{F^n_*\Omega^{m}_Y(\log E)}{B_n\Omega^{m}_Y(\log E)} \ar{r} & 0.
\end{tikzcd}
\]
This concludes the proof of the theorem. \qedhere
\end{proof}

In what follows, we explain how to use the above result and reduction modulo $p \gg 0$ to show that singularities of pre-$m$-$F$-rational type are pre-$m$-rational.

\begin{theorem}\label{thm:k-F-rational to k-rational}
    Fix an integer $m \geq 0$ and let $X$ be a normal affine variety over a field $K$ of characteristic zero satisfying $\codim_X {\rm Sing}(X) >m$. Let $X_A$ be a flat model of $X$ over a finitely generated $\bZ$-subalgebra $A \subseteq K$. 

    Suppose there exists a Zariski dense set of closed points $S\subseteq \Spec A$ such that 
    $X_s$ is pre-$m$-$F$-rational for every $s\in S$. 
    Then $X$ is pre-$m$-rational.
\end{theorem}
\begin{proof}
We prove the assertion by ascending induction on $m$.
The case $m=0$ has been proven by Smith \cite{Smith(rational)}, and thus we may assume that $m\geq 1$.
Let $\pi\colon Y\to X$ be a strong log resolution.
Since we can assume that $X$ is pre-$(m-1)$-rational, Proposition \ref{prop:characterization of pre-k-rational} yields 
\[
    R^{>0}\pi_{*}\Omega^i_Y(\log E)=0
\] for all $i\leq m-1$, and 
\[
\Omega^i_{X,h}=\pi_{*}\Omega^i_Y(\log E)=\Omega^{[i]}_X
\]
for all $i\geq 0$ (see \cite[Proposition 2.29]{Kawakami-Witaszek(ch=0)}).

In particular, by Lemma \ref{lem:Hara's lemma} and Corollary \ref{cor:reduction(local-coh)-Omega},  we can replace $\Spec A$ by an affine open subset so that
\begin{equation}\label{eq:(k-1)-rational assumption}
R^{>0}\pi_{s,*}\Omega^i_{Y_s}(\log E_s)=0
\end{equation}
for all $i\leq m-1$, and
\[
\pi_{s,*}\Omega^i_{Y_s}(\log E_s)=\Omega^{[i]}_{X_s}
\]
for every closed point $s \in S$.

Restricting ourselves to integers $i,j \geq 0$ such that $i \leq m$ and $i+j \leq d$, Corollary \ref{cor:reduction(local-coh)-Omega} implies that there exists an open subset $U \subseteq \Spec A$ such that the map
     \begin{equation} \label{eq:rationalDesiredInjective}
      \pi^* \colon H^j_\m(\Omega^{[i]}_X) \to H^j_\m(R\pi_{*}\Omega^i_Y(\log E))
     \end{equation}
     is injective  for every closed point $\m \in X$ if and only if there exists a closed point $s \in U$ such that the map
     \[
     \pi^*_{s} \colon H^j_\n(\Omega^{[i]}_{X_s}) \to H^j_\n(R\pi_{s,*}\Omega^i_{Y_s}(\log E_s))
     \]
     is injective for every closed point $\n \in X_s$. 
     
     Since there exists at least one closed point $s \in S \cap U \neq \emptyset$, the latter condition is satisfied when $i \leq m$ and $i+j \leq d$ by Theorem \ref{thm:main-component-to-char-0} in view of (\ref{eq:(k-1)-rational assumption}) and $X_s$ being pre-$m$-$F$-rational. Hence the former condition (\ref{eq:rationalDesiredInjective}) is satisfied when $i \leq m$ and $i+j \leq d$, which is equivalent to $X$ being pre-$m$-rational by Theorem \ref{thm:characterization of pre-k-rational}. This concludes the proof. \qedhere
\end{proof}

\subsection{$m$-rational implies $m$-$F$-rational}
In this subsection, we show that the reduction of an $m$-rational singularity modulo every prime number $p \gg 0$ is $m$-$F$-rational. 

We start with the following lemma.
\begin{lemma} \label{lem:CSurj}
     Let $X$ be a normal affine $d$-dimensional variety over a perfect field of characteristic $p>d$, and let $\pi\colon Y\to X$ be a log resolution with the reduced exceptional divisor $E$ supporting a $\pi$-ample $\bQ$-divisor $A$ such that $\rup{A}=0$.
    Suppose:
    \begin{enumerate}
        \item $R^{d-i}\pi_*\Omega^i_Y(\log E)(p^lA)=0$ for all $l>0$ and $i < d$, and \\[-0.8em]
        \item the snc pair $(Y,E)$ lifts to $W_2(k)$.
    \end{enumerate}
    Then the map 
    \[
    R^j\pi_{*}Z_{n}\Omega^i_Y(\log E)(p^nA)
        \xrightarrow{C_n} 
        R^j\pi_{*}\Omega^i_Y(\log E)(A)
    \]
        is surjective for all $i,j,n$ such that $i+j \geq d$ and $n>0$.
\end{lemma}
\begin{proof}
By (\ref{BZOmega}), it suffices to prove that 
        \[
        R^{j+1}\pi_{*}B_{n}\Omega^i_Y(\log E)(p^nA)=0.
        \]
        In turn, by the short exact sequence
        \[
           0\to F^{n-1}_{*}B\Omega^i_Y(\log E)(p^lA)\to B_{l}\Omega^i_Y(\log E)(p^lA) \to B_{l-1}\Omega^i_Y(\log E)(p^{l-1}A) \to 0, 
        \]
        it suffices to show that 
        \[
        R^{j+1}\pi_{*}B\Omega^i_Y(\log E)(p^lA)=0
        \]
        for all $l>0$.
        Next, using the short exact sequences
        \begin{align*}
            &0\to Z\Omega^{i-1}_Y(\log E)(p^lA)\to F_{*}\Omega^{i-1}_Y(\log E)(p^lA)\to B\Omega^i_Y(\log E)(p^{l}A) \to 0\\
            &0\to B\Omega^{i-1}_Y(\log E)(p^lA)\to Z\Omega^{i-1}_Y(\log E)(p^lA)\to \Omega^{i-1}_Y(\log E)(p^{l-1}A) \to 0,
        \end{align*}
        we reduce recursively to proving that:\\[-0.8em]
        \begin{enumerate}
            \item[(a)] $R^{j}\pi_*\Omega^i_Y(\log E)(p^lA)=0$ when $l>0$, $i<d$, and $i+j = d$ \\[-0.8em]
            \item[(b)] $R^{j}\pi_*\Omega^i_Y(\log E)(p^{l-1}A)=0$ when $l>0$, $i<d$, and $i+j > d$.\\[-0.8em]
        \end{enumerate}
        Assertion (a) follows from Assumption (1), and Assertion (b) follows from Assumption (2) and the Akizuki-Nakano vanishing for $W_2(k)$-liftable pairs \cite[Corollary 3.8]{Hara98} since $p>d$. 
\end{proof}

The following result is the key component of the proof of the main theorem of this subsection. 

\begin{theorem}\label{thm:pre-k-Du Bois to pre-$F$-inj in p>0}
    Let $X$ be a normal $d$-dimensional affine variety over a perfect field of characteristic $p>d\coloneqq\dim X$, and let $\pi\colon Y\to X$ be a log resolution with the reduced exceptional divisor $E$ supporting a $\pi$-ample divisor $A$ such that $\rup{A}=0$ and $\Supp\, \{p^n A\} = E$ for all $n \geq 0$.
    Suppose that $X$ is $F$-rational and that for a fixed integer $m>0$:
    \begin{enumerate}
            \item $R\pi_{*}\Omega^i_Y(\log E) \cong \Omega^{[i]}_X$ for all $i\leq m$, \\[-0.8em]
        \item $R^{d-i}\pi_*\Omega^i_Y(\log E)(p^lA)=0$ for all $l>0$ and $i < d$, \\[-0.8em]
        \item the snc pair $(Y,E)$ lifts to $W_2(k)$.
    \end{enumerate}
    Then $X$ is pre-$m$-$F$-rational. 
\end{theorem}
\begin{proof}
By Lemma \ref{lemma:Kawakami-vanishing} and Assumption (1), we get that $C\colon Z\Omega^{[i]}_{X}\to \Omega^{[i]}_{X}$ is surjective for all $i\leq m$ and
\[
R\pi_*B_n\Omega^i_Y(\log E) = B_n\Omega^{[i]}_X
\]
for all $i \leq m$. 
Now, by Theorem \ref{thm:IndependenceIntro} (3), it it remains to show that there exists an integer $n_0>0$ such that the composition
\begin{equation} \label{eq:tocharpGoal}
H^j_\m(\Omega^{[i]}_{X}) \xrightarrow{C_n^{-1}} H^j_\m\left(\frac{F^n_{*}\Omega^{[i]}_{X}}{B_n\Omega^{[i]}_{X}}\right) \xrightarrow{\delta} H^{j+1}_\m(B_n\Omega^{[i]}_X) 
\end{equation}
is injective for all $i, j, n$ such that
$0 < i \leq m$,  $i+j \leq d$, and $n \geq n_0$. 

To this end, consider the following commutative diagram for $i \leq m$:
\[
\begin{tikzcd}[column sep = small]
H_{\m}^j(\Omega^{[i]}_{X}) \arrow[r,"C^{-1}_{n}"]\arrow{d}{=} & H_{\m}^j\left(\frac{F^n_{*}\Omega_X^{[i]}}{B_n\Omega^{[i]}_{X}}\right) \arrow{d}{=} \ar{r}{\delta} & H_{\m}^{j+1}(B_n\Omega^{[i]}_{X})  \ar{d}{=} \\
H_{\m}^j(R\pi_{*}\Omega^i_{Y}(\log E)) \ar{d}{=} \arrow[r,"C^{-1}_{n}"]& H_{\m}^j\left(R\pi_{*}\frac{F^n_*\Omega^i_{Y}(\log E)}{B_n\Omega^i_Y(\log E)}\right) \ar{d} \ar{r}{R\pi_*\delta_Y} & H_{\m}^{j+1}(R\pi_*B_n\Omega^i_Y(\log E)) \ar{d} \\
H_{\m}^j(R\pi_{*}\Omega^i_{Y}(\log E)(-A)) \arrow[r,"C^{-1}_{n}"]& H_{\m}^j\left(R\pi_{*}\frac{F^n_*\Omega^i_{Y}(\log E)(-p^nA)}{B_n\Omega^i_Y(\log E)(-p^nA)}\right) \ar{r} & H_{\m}^{j+1}(R\pi_*B_n\Omega^i_Y(\log E)(-p^nA)).
\end{tikzcd}
\]

\begin{claim} \label{claim:tocharp} The map
\[
H_{\m}^j(R\pi_{*}\Omega^i_{Y}(\log E)(-A)) \xrightarrow{C^{-1}_{n}} H_{\m}^j\left(R\pi_{*}\frac{F^n_*\Omega^i_{Y}(\log E)(-p^nA)}{B_n\Omega^i_Y(\log E)(-p^nA)}\right)
\]
is injective for all $i, j, n \geq 0$ such that $1 \leq i \leq m$, $i+j \leq d$, and $n \gg 0$.
\end{claim}

Note that this claim implies our goal (\ref{eq:tocharpGoal}). Indeed, this follows by tracing through the following exact sequence:
\begin{align*} \label{eq:independentBigDiagramToCharP} \nonumber
0 = H_{\m}^j(R\pi_{*}F^n_*\Omega^i_{Y}(\log E)(-p^nA)) \to& \ H_{\m}^j\left(R\pi_{*}\frac{F^n_*\Omega^i_{Y}(\log E)(-p^nA)}{B_n\Omega^i_Y(\log E)(-p^nA)}\right) \\ \xrightarrow{R\pi_*\delta_{Y,-p^nA}}& \ H^{j+1}_\m(R\pi_*B_n\Omega^{i}_Y(\log E)(-p^nA))
\end{align*}
in which the vanishing on the left holds by local duality and Serre's vanishing as $j<d$ and we can assume that $n \gg 0$.\\

Next, we move to the proof of Claim \ref{claim:tocharp}. 
By local duality and Lemma \ref{lem:dual} below,
     the map from the claim becomes:
\[
C_n \colon R^{d-j}\pi_*Z_n\Omega^{d-i}_Y(\log E)(-E+\rup{p^nA}) \to R^{d-j}\pi_*\Omega^{d-i}_Y(\log E)(-E+\rup{A}).
\] 
Replacing $d-i$ and $d-j$ with $i$ and $j$ respectively, 
and observing that $-E + \rup{A} = \rdown{A}$ and $-E + \rup{p^nA} = \rdown{p^nA}$ thanks to $\Supp\, \{p^n A\} = E$, our problem reduces to showing that
\[
R^j\pi_{*}Z_{n}\Omega^i_Y(\log E)(p^nA)
        \xrightarrow{C_n} 
        R^j\pi_{*}\Omega^i_Y(\log E)(A)
\]
is surjective for all $i\geq d-m$ and $i + j \geq d$. This follows from Lemma \ref{lem:CSurj}.

\end{proof}

\begin{lemma}\label{lem:dual}
    Let $Y$ be a smooth variety over a perfect field of characteristic $p>0$, and let $E$ be a simple normal crossing divisor.
    Let $A$ be a $\Q$-divisor such that $\Supp\, \{p^n A\} = E$ for all $n \geq 0$.
    Then \begin{equation}\label{dual1}
        \cHom\left(\frac{F^n_*\Omega^i_{Y}(\log E)(-p^nA)}{B_n\Omega^i_Y(\log E)(-p^nA)},\omega_Y\right)\cong Z_n\Omega^{d-i}_Y(\log E)(-E+\rup{p^nA})
    \end{equation}
     for all $i\geq 0$ and $n\geq 1$.
\end{lemma}
\begin{proof}
    We prove the assertion by ascending induction on $n$.
    By the definition of $B\Omega^{i}_Y(\log E)(-pA)$, we have an exact sequence
    \[
    F_{*}\Omega^{i-1}_Y(\log E)(-pA)\xrightarrow{F_{*}d} F_{*}\Omega^{i}_Y(\log E)(-pA)\to \frac{F_{*}\Omega^{i}_Y(\log E)(-pA)}{B\Omega^{i}_Y(\log E)(-pA)}\to 0.
    \]
    By taking $\cHom(-,\omega_Y)$, we get an exact sequence
    \begin{multline*}
        0\to \cHom\left(\frac{F_{*}\Omega^{i}_Y(\log E)(-pA)}{B\Omega^{i}_Y(\log E)(-pA)},\omega_Y\right)\to 
    F_{*}\Omega^{d-i}_Y(\log E)(-E+\rup{pA})\\\xrightarrow{F_*d} F_{*}\Omega^{d-i+1}_Y(\log E)(-E+\rup{pA}). 
    \end{multline*}
    By the definition of $Z\Omega^{d-i}_Y(\log E)(-E+\rup{pA})$, we then obtain \eqref{dual1} for $n=1$.

    Next, consider the short exact sequence
    \[
    0\to Z\Omega^{i-1}_Y(\log E)(-pA)\to F_{*}\Omega^{i-1}_Y(\log E)(-pA)\to B\Omega^{i}_Y(\log E)(-pA)\to 0.
    \]
    By taking $\cHom(-,\omega_Y)$, we get an exact sequence
    \begin{multline*}
        0\to \cHom\left(B\Omega^{i}_Y(\log E)(-pA),\omega_Y\right)\to 
    F_{*}\Omega^{d-i+1}_Y(\log E)(-E+\rup{pA})\\\to \frac{F_{*}\Omega^{d-i+1}_Y(\log E)(-E+\rup{pA})}{B\Omega^{d-i+1}_Y(\log E)(-E+\rup{pA})}\to 0,
    \end{multline*}
    where we used the case $n=1$ of \eqref{dual1} in the last term.
    Thus 
    \[
    \cHom\left(B\Omega^{i}_Y(\log E)(-pA),\omega_Y\right)\cong B\Omega^{d-i+1}_Y(\log E)(-E+\rup{pA}).
    \]
    By \eqref{eq:BBB2}, we have an isomorphism
    \[
    C_{n-1}^{-1} \colon B\Omega^i_Y(\log E)(-pA)  \xrightarrow{\cong} \left(\frac{B_n\Omega^i_Y(\log E)(-p^{n}A)}{F_{*}B_{n-1}\Omega^i_Y(\log E)(-p^{n}A)}\right).
    \]
    Thus, we obtain a short exact sequence
    \begin{multline*}
        0\to B\Omega^i_Y(\log E)(-pA) \xrightarrow{C_{n-1}^{-1}} F_{*}\left(\frac{F^{n-1}_{*}\Omega^i_Y(\log E)(-p^{n}A)}{B_{n-1}\Omega^i_Y(\log E)(-p^{n}A)}\right)\\\to \frac{F^{n}_{*}\Omega^i_Y(\log E)(-p^{n}A)}{B_{n}\Omega^i_Y(\log E)(-p^{n}A)}\to 0.
    \end{multline*}
    
    By taking the dual, we get
    \begin{multline*}
        0\to \cHom\left(\frac{F^{n}_{*}\Omega^i_Y(\log E)(-p^{n}A)}{B_{n}\Omega^i_Y(\log E)(-p^{n}A)},\omega_Y\right) \to F_*Z_{n-1}\Omega^{d-i}_Y(\log E)(-E+\rup{p^nA}) \\
        \to B\Omega^{d-i+1}_Y(\log E)(-E+\rup{pA})\to 0,
    \end{multline*}
    where we used the induction hypothesis for the middle term. 
    Since $\Supp\, \{p^i A\} = E$ for every $i\geq 0$, we have $\rdown{p^iA}=-E+\rup{p^iA}$. 
    Thus, by \eqref{ZZB}, we conclude that
    \[
    \cHom\left(\frac{F^{n}_{*}\Omega^{i}_Y(\log E)(-p^{n}A)}{B_{n}\Omega^i_Y(\log E)(-p^{n}A)},\omega_Y\right)\cong Z_{n}\Omega^{d-i}_Y(\log E)(-E+\lceil p^nA \rceil),
    \]
    as desired.
\end{proof}

\begin{theorem}\label{thm:pre-k-F-rational type to pre-$k$-rational}
    Let $X$ be a normal affine variety over an algebraically closed field $K$ of characteristic zero with pre-$m$-rational singularities for an integer $m < \codim_X {\rm Sing}(X)$.
    Then,
    given a model of $X$ over a finitely generated $\mathbb{Z}$-subalgebra $A$ of $K$,
    there exists a non-empty open subset $S\subseteq \Spec A$ such that 
    $X_s$ is pre-$m$-$F$-rational for all closed points $s\in S$.
\end{theorem}
\begin{proof} 
      Take a strong log resolution $\pi\colon Y\to X$ with the reduced exceptional divisor $E$ supporting a $\pi$-ample divisor. 
      We may assume that $X$ is rational and that $\pi_{*}\Omega^i_Y(\log E)=\Omega^{[i]}_X$ for all $i\geq0$ (see \cite{KS21}).
      Fix a $\pi$-ample divisor $D$ such that $\rup{D}=0$.
    By Proposition \ref{prop:characterization of pre-k-rational}, we have that $R^{>0}\pi_{*}\Omega^i_{Y}(\log E)=0$ for all $i\leq m$.
    By shrinking $\Spec A$, we can take a 
    model \[
    (\pi_A\colon Y_A\to X_A,\ E_A,\ D_A) \quad\text{of}\quad
    (\pi\colon Y\to X,\ E,\ D)
    \] 
    so that the following conditions are satisfied:
    \begin{enumerate}
        \item $\Exc(\pi_A)=E_A$ and $(Y_A, E_A)$ is a simple normal crossing pair over $\Spec A$. In particular, $\Omega^i_{Y_A}(\log E_A)$ is free over $A$ for all $i\geq 0$.
        \item $R^{>0}\pi_{A,*}\Omega^i_{Y_A}(\log E_A)=0$ for all $i\leq m$.
        \item $X_s$ is $F$-rational for all closed points $s \in \Spec A$ (see \cite[Theorem 1.1]{Mehta-Srinivas(rationalsingularities)} or \cite[Theorem 1.1]{Hara98}).
        \item $\pi_{s,*}\Omega^i_{Y_s}(\log E_s)\cong \Omega^{[i]}_{X_s}$ (see Corollary \ref{cor:reduction(local-coh)-Omega}).  
    \end{enumerate}
    By Serre vanishing, we can take an integer $n_0 \gg 0$ such that 
    \begin{equation}\label{eq:reduction2}
    R^{>0}\pi_{A,*}\Omega^i_{Y_A}(\log E_{A})(nD_{A})=0
     \end{equation}
    for all $i\geq 0$ and $n \geq n_0$. By replacing $\Spec A$ by an open subset, we may assume that for all closed points $s\in \Spec A$, we have $p_s\coloneq \mathrm{char}(\kappa(s))>n_0$. We can also assume that $p_s$ is bigger than $\dim X$ and the absolute values of all integers appearing in the denominators of the coefficients of $D$, in particular, $\Supp \{p_s^nD_s\}=E_s$ holds for all $s$ and $n\geq 0$.
   
    By combining Lemma \ref{lem:Hara's lemma}, (2), and (4)
    we obtain
    \[
    R\pi_{s,*}\Omega^i_{Y_s}(\log E_s) \cong \Omega^{[i]}_{X_s} 
    \]
    for all $i\leq m$. Moreover, by Lemma \ref{lem:Hara's lemma} and \eqref{eq:reduction2}, we obtain
    \[
    R^{>0}\pi_{s,*}\Omega^i_{Y_s}(\log E_s)(nD_s)=0
    \]
    for all $n \geq n_0$.
    In particular, we have
        \[
    R^{>0}\pi_{s,*}\Omega^i_{Y_s}(\log E_s)(p_s^{l}D_s)=0
    \]
    for all $l> 0$ and for all closed points $s\in \Spec A$ since $p_s>n_0$.
    Finally, the snc pair $(Y_s, E_s)$ lifts to $W_2(k(s))$ by \cite[Proposition 2.5]{ABL} as we can assume that $A$ is smooth over $\bZ$.
    
    Therefore  the assumptions of Theorem \ref{thm:pre-k-Du Bois to pre-$F$-inj in p>0} are satisfied and so $X_s$ is pre-$m$-$F$-rational for every closed point $s \in \Spec A$ as required. \qedhere
\end{proof}

\input{main.bbl}

% \bibliographystyle{skalpha}
% \bibliography{bibliography.bib}

\end{document}

%% file: main.bbl

%% file: main.bbl
\begin{bibdiv}
\begin{biblist}

\bib{ABL}{article}{
      author={Arvidsson, Emelie},
      author={Bernasconi, Fabio},
      author={Lacini, Justin},
       title={On the {K}awamata-{V}iehweg vanishing theorem for log del {P}ezzo surfaces in positive characteristic},
        date={2022},
        ISSN={0010-437X},
     journal={Compos. Math.},
      volume={158},
      number={4},
       pages={750\ndash 763},
         url={https://doi.org/10.1112/S0010437X22007394},
      review={\MR{4438290}},
}

\bib{Artin(approximation)}{article}{
      author={Artin, M.},
       title={Algebraic approximation of structures over complete local rings},
        date={1969},
        ISSN={0073-8301},
     journal={Inst. Hautes \'{E}tudes Sci. Publ. Math.},
      number={36},
       pages={23\ndash 58},
         url={http://www.numdam.org/item?id=PMIHES_1969__36__23_0},
      review={\MR{268188}},
}

\bib{CDM24}{article}{
      author={Chen, Qianyu},
      author={Dirks, Bradley},
      author={Musta\c{t}\u{a}, Mircea},
       title={The minimal exponent and {$k$}-rationality for local complete intersections},
        date={2024},
        ISSN={2429-7100,2270-518X},
     journal={J. \'Ec. polytech. Math.},
      volume={11},
       pages={849\ndash 873},
         url={https://doi-org.utokyo.idm.oclc.org/10.5802/jep.267},
      review={\MR{4791993}},
}

\bib{CLS11}{book}{
      author={Cox, David~A.},
      author={Little, John~B.},
      author={Schenck, Henry~K.},
       title={Toric varieties},
      series={Graduate Studies in Mathematics},
   publisher={American Mathematical Society, Providence, RI},
        date={2011},
      volume={124},
        ISBN={978-0-8218-4819-7},
         url={https://doi-org.turing.library.northwestern.edu/10.1090/gsm/124},
      review={\MR{2810322}},
}

\bib{Friedman-Laza2}{article}{
      author={Friedman, Robert},
      author={Laza, Radu},
       title={The higher {D}u {B}ois and higher rational properties for isolated singularities},
        date={2024},
        ISSN={1056-3911,1534-7486},
     journal={J. Algebraic Geom.},
      volume={33},
      number={3},
       pages={493\ndash 520},
      review={\MR{4739666}},
}

\bib{Friedman-Laza1}{article}{
      author={Friedman, Robert},
      author={Laza, Radu},
       title={Higher {D}u {B}ois and higher rational singularities},
        date={2024},
        ISSN={0012-7094,1547-7398},
     journal={Duke Math. J.},
      volume={173},
      number={10},
       pages={1839\ndash 1881},
         url={https://doi.org/10.1215/00127094-2023-0051},
        note={Appendix by Morihiko Saito},
      review={\MR{4776417}},
}

\bib{GNPP}{book}{
      author={Guill\'{e}n, F.},
      author={Navarro~Aznar, V.},
      author={Pascual~Gainza, P.},
      author={Puerta, F.},
       title={Hyperr\'{e}solutions cubiques et descente cohomologique},
      series={Lecture Notes in Mathematics},
   publisher={Springer-Verlag, Berlin},
        date={1988},
      volume={1335},
        ISBN={3-540-50023-5},
         url={https://doi-org.kyoto-u.idm.oclc.org/10.1007/BFb0085054},
        note={Papers from the Seminar on Hodge-Deligne Theory held in Barcelona, 1982},
      review={\MR{972983}},
}

\bib{Gra}{article}{
      author={Graf, Patrick},
       title={Differential forms on log canonical spaces in positive characteristic},
        date={2021},
        ISSN={0024-6107},
     journal={J. Lond. Math. Soc. (2)},
      volume={104},
      number={5},
       pages={2208\ndash 2239},
      review={\MR{4368674}},
}

\bib{Hara98}{article}{
      author={Hara, Nobuo},
       title={A characterization of rational singularities in terms of injectivity of {F}robenius maps},
        date={1998},
        ISSN={0002-9327},
     journal={Amer. J. Math.},
      volume={120},
      number={5},
       pages={981\ndash 996},
         url={http://muse.jhu.edu/journals/american_journal_of_mathematics/v120/120.5hara.pdf},
      review={\MR{1646049}},
}

\bib{Huber-Jorder}{article}{
      author={Huber, Annette},
      author={J\"{o}rder, Clemens},
       title={Differential forms in the h-topology},
        date={2014},
        ISSN={2313-1691,2214-2584},
     journal={Algebr. Geom.},
      volume={1},
      number={4},
       pages={449\ndash 478},
         url={https://doi.org/10.14231/AG-2014-020},
      review={\MR{3272910}},
}

\bib{JKSY22}{article}{
      author={Jung, Seung-Jo},
      author={Kim, In-Kyun},
      author={Saito, Morihiko},
      author={Yoon, Youngho},
       title={Higher {D}u {B}ois singularities of hypersurfaces},
        date={2022},
        ISSN={0024-6115,1460-244X},
     journal={Proc. Lond. Math. Soc. (3)},
      volume={125},
      number={3},
       pages={543\ndash 567},
         url={https://doi.org/10.1112/plms.12464},
      review={\MR{4480883}},
}

\bib{Kaw4}{misc}{
      author={Kawakami, Tatsuro},
       title={Extendability of differential forms via {C}artier operators},
         how={\url{https://arxiv.org/abs/2207.13967v4}},
        date={2022},
        note={To appear in \emph{J.~Eur.~Math.~Soc.~(JEMS)}},
}

\bib{Kaw7}{article}{
      author={Kawakami, Tatsuro},
       title={On {S}teenbrink vanishing for rational singularities in positive characteristic},
        date={2025},
     journal={arxiv:2507.04838},
}

\bib{KS21}{article}{
      author={Kebekus, Stefan},
      author={Schnell, Christian},
       title={Extending holomorphic forms from the regular locus of a complex space to a resolution of singularities},
        date={2021},
        ISSN={0894-0347},
     journal={J. Amer. Math. Soc.},
      volume={34},
      number={2},
       pages={315\ndash 368},
         url={https://doi.org/10.1090/jams/962},
      review={\MR{4280862}},
}

\bib{Kawakami-Sato}{article}{
      author={Kawakami, Tatsuro},
      author={Sato, Kenta},
       title={Extending one-forms on {$F$}-regular singularities},
        date={2025},
     journal={arXiv:2502.17148},
}

\bib{Kawakami-Takamatsu}{article}{
      author={Kawakami, Tatsuro},
      author={Takamatsu, Teppei},
       title={On {F}robenius liftability of surface singularities},
        date={2024},
     journal={arXiv:2402.08152},
         url={https://arxiv.org/abs/2402.08152},
}

\bib{KTTWYY1}{article}{
      author={Kawakami, Tatsuro},
      author={Takamatsu, Teppei},
      author={Tanaka, Hiromu},
      author={Witaszek, Jakub},
      author={Yobuko, Fuetaro},
      author={Yoshikawa, Shou},
       title={Quasi-{$F$}-splittings in birational geometry},
        date={2025},
        ISSN={0012-9593,1873-2151},
     journal={Ann. Sci. \'Ec. Norm. Sup\'er. (4)},
      volume={58},
      number={3},
       pages={665\ndash 748},
      review={\MR{4962159}},
}

\bib{Kawakami-Witaszek}{article}{
      author={Kawakami, Tatsuro},
      author={Witaszek, Jakub},
       title={Higher {F}-injective singularities},
        date={2024},
     journal={arXiv:2412.08887},
}

\bib{Kawakami-Witaszek(ch=0)}{article}{
      author={Kawakami, Tatsuro},
      author={Witaszek, Jakub},
       title={Inversion of adjunction for higher rational singularities},
        date={2025},
     journal={arXiv:2510.11378},
}

\bib{Langer19}{article}{
      author={Langer, Adrian},
       title={Birational geometry of compactifications of {D}rinfeld half-spaces over a finite field},
        date={2019},
        ISSN={0001-8708},
     journal={Adv. Math.},
      volume={345},
       pages={861\ndash 908},
         url={https://doi.org/10.1016/j.aim.2019.01.031},
      review={\MR{3902334}},
}

\bib{MOPW23}{article}{
      author={Musta\c{t}\u{a}, Mircea},
      author={Olano, Sebasti\'an},
      author={Popa, Mihnea},
      author={Witaszek, Jakub},
       title={The {D}u {B}ois complex of a hypersurface and the minimal exponent},
        date={2023},
     journal={Duke Math. J.},
      volume={172},
      number={7},
       pages={1411\ndash 1436},
}

\bib{Mustata-Popa22}{article}{
      author={Musta\c{t}\u{a}, Mircea},
      author={Popa, Mihnea},
       title={Hodge filtration on local cohomology, {D}u {B}ois complex and local cohomological dimension},
        date={2022},
     journal={Forum Math. Pi},
      volume={10},
       pages={Paper No. e22, 58},
         url={https://doi-org.kyoto-u.idm.oclc.org/10.1017/fmp.2022.15},
      review={\MR{4491455}},
}

\bib{Mehta-Srinivas(rationalsingularities)}{article}{
      author={Mehta, V.~B.},
      author={Srinivas, V.},
       title={A characterization of rational singularities},
        date={1997},
        ISSN={1093-6106,1945-0036},
     journal={Asian J. Math.},
      volume={1},
      number={2},
       pages={249\ndash 271},
         url={https://doi-org.utokyo.idm.oclc.org/10.4310/AJM.1997.v1.n2.a4},
      review={\MR{1491985}},
}

\bib{Park-Popa1}{article}{
      author={Park, Sung~Gi},
      author={Popa, Mihnea},
       title={Hodge symmetry and {L}efschetz theorems for singular varieties},
        date={2025},
     journal={arXiv:2410.15638},
}

\bib{Park-Popa2}{article}{
      author={Park, Sung~Gi},
      author={Popa, Mihnea},
       title={Q-factoriality and {H}odge-{D}u {B}ois theory},
        date={2025},
     journal={arXiv:2508.17748},
}

\bib{Peter-Steenbrink(Book)}{book}{
      author={Peters, Chris A.~M.},
      author={Steenbrink, Joseph H.~M.},
       title={Mixed {H}odge structures},
      series={Ergebnisse der Mathematik und ihrer Grenzgebiete. 3. Folge. A Series of Modern Surveys in Mathematics [Results in Mathematics and Related Areas. 3rd Series. A Series of Modern Surveys in Mathematics]},
   publisher={Springer-Verlag, Berlin},
        date={2008},
      volume={52},
        ISBN={978-3-540-77015-2},
      review={\MR{2393625}},
}

\bib{popa2024injectivityvanishingdubois}{article}{
      author={Popa, Mihnea},
      author={Shen, Wanchun},
      author={Vo, Anh~Duc},
       title={Injectivity and {V}anishing for the {D}u {B}ois {C}omplexes of {I}solated {S}ingularities},
        date={2024},
     journal={arXiv:2409.18019},
         url={https://arxiv.org/abs/2409.18019},
}

\bib{Smith(rational)}{article}{
      author={Smith, Karen~E.},
       title={{$F$}-rational rings have rational singularities},
        date={1997},
        ISSN={0002-9327,1080-6377},
     journal={Amer. J. Math.},
      volume={119},
      number={1},
       pages={159\ndash 180},
         url={http://muse.jhu.edu.utokyo.idm.oclc.org/journals/american_journal_of_mathematics/v119/119.1smith.pdf},
      review={\MR{1428062}},
}

\bib{stacks-project}{misc}{
      author={{Stacks Project Authors}, The},
       title={\itshape {S}tacks {P}roject},
         how={\url{http://stacks.math.columbia.edu}},
        date={2026},
}

\bib{SVV}{article}{
      author={Shen, Wanchun},
      author={Venkatesh, Sridhar},
      author={Vo, Anh~Duc},
       title={On $ k $-{D}u {B}ois and $ k $-rational singularities},
        date={2023},
     journal={arXiv preprint arXiv:2306.03977},
}

\bib{Takagi-Watanabe}{article}{
      author={Takagi, Shunsuke},
      author={Watanabe, Kei-Ichi},
       title={{$F$}-singularities: applications of characteristic {$p$} methods to singularity theory [translation of {MR}3135334]},
        date={2018},
        ISSN={0898-9583},
     journal={Sugaku Expositions},
      volume={31},
      number={1},
       pages={1\ndash 42},
         url={https://doi.org/10.1090/suga/427},
      review={\MR{3784697}},
}

\end{biblist}
\end{bibdiv}
